\documentclass[11pt,a4paper,reqno]{amsart}
\usepackage[english]{babel}
\usepackage[T1]{fontenc}
\usepackage{verbatim}
\usepackage{palatino}
\usepackage{amsmath}
\usepackage{mathabx}
\usepackage{amssymb}
\usepackage{amsthm}
\usepackage{amsfonts}
\usepackage{graphicx}
\usepackage{esint}
\usepackage{color}
\usepackage{mathtools}
\usepackage{overpic}

\usepackage[colorlinks = true, citecolor = black]{hyperref}
\author{Tuomas Orponen}
\title{Additive properties of fractal sets on the parabola}
\address{Department of Mathematics and Statistics\\ University of Jyv\"askyl\"a,
P.O. Box 35 (MaD)\\
FI-40014 University of Jyv\"askyl\"a\\
Finland}
\email{tuomas.t.orponen@jyu.fi}
\date{\today}
\thanks{T.O. is supported by the Academy of Finland via the projects \emph{Quantitative rectifiability in Euclidean and non-Euclidean spaces} and \emph{Incidences on Fractals}, grant Nos. 309365, 314172, 321896.}

\newcommand{\R}{\mathbb{R}}

\newcommand{\He}{\mathbb{H}}
\newcommand{\N}{\mathbb{N}}

\newcommand{\C}{\mathbb{C}}
\newcommand{\Z}{\mathbb{Z}}

\newcommand{\spt}{\operatorname{spt}}
\newcommand{\Hd}{\dim_{\mathrm{H}}}

\newcommand{\spa}{\operatorname{span}}
\newcommand{\diam}{\operatorname{diam}}

\newcommand{\dist}{\operatorname{dist}}

\def\Barint_#1{\mathchoice
          {\mathop{\vrule width 6pt height 3 pt depth -2.5pt
                  \kern -8pt \intop}\nolimits_{#1}}%
          {\mathop{\vrule width 5pt height 3 pt depth -2.6pt
                  \kern -6pt \intop}\nolimits_{#1}}%
          {\mathop{\vrule width 5pt height 3 pt depth -2.6pt
                  \kern -6pt \intop}\nolimits_{#1}}%
          {\mathop{\vrule width 5pt height 3 pt depth -2.6pt
                  \kern -6pt \intop}\nolimits_{#1}}}

\numberwithin{equation}{section}

\theoremstyle{plain}
\newtheorem{thm}[equation]{Theorem}
\newtheorem*{"thm"}{"Theorem"}
\newtheorem{conjecture}[equation]{Conjecture}
\newtheorem{lemma}[equation]{Lemma}

\newtheorem{ex}[equation]{Example}
\newtheorem{cor}[equation]{Corollary}
\newtheorem{proposition}[equation]{Proposition}

\theoremstyle{definition}

\newtheorem{definition}[equation]{Definition}

\theoremstyle{remark}
\newtheorem{remark}[equation]{Remark}

\newcommand{\nref}[1]{(\hyperref[#1]{#1})}

\DeclareMathSymbol{\intop}  {\mathop}{mathx}{"B3}

\begin{document} 

\begin{abstract} Let $0 \leq s \leq 1$, and let $\mathbb{P} := \{(t,t^{2}) \in \R^{2} : t \in [-1,1]\}$. If $K \subset \mathbb{P}$ is a closed set with $\Hd K = s$, it is not hard to see that $\Hd (K + K) \geq 2s$. The main corollary of the paper states that if $0 < s < 1$, then adding $K$ once more makes the sum slightly larger:
\begin{displaymath} \Hd (K + K + K) \geq 2s + \epsilon, \end{displaymath}
where $\epsilon = \epsilon(s) > 0$. This information is deduced from an $L^{6}$ bound for the Fourier transforms of Frostman measures on $\mathbb{P}$. If $0 < s < 1$, and $\mu$ is a Borel measure on $\mathbb{P}$ satisfying $\mu(B(x,r)) \leq r^{s}$ for all $x \in \mathbb{P}$ and $r > 0$, then there exists $\epsilon = \epsilon(s) > 0$ such that
\begin{displaymath} \|\hat{\mu}\|_{L^{6}(B(R))}^{6} \leq R^{2 - (2s + \epsilon)} \end{displaymath} 
for all sufficiently large $R \geq 1$. The proof is based on a reduction to a $\delta$-discretised point-circle incidence problem, and eventually to the $(s,2s)$-Furstenberg set problem.
\end{abstract}

\maketitle

\tableofcontents

\section{Introduction}

The main result of this paper investigates the $L^{p}$-norms of Fourier transforms of fractal measures on the truncated parabola $\mathbb{P} = \{(t,t^{2}) : t \in [-1,1]\}$:\begin{thm}\label{main} Let $0 \leq s \leq 1$, and let $\mu$ be a Borel measure on $\mathbb{P}$ satisfying the Frostman condition $\mu(B(x,r)) \leq r^{s}$ for all $x \in \R^{2}$ and $r > 0$. Then,
\begin{equation}\label{form89} \|\hat{\mu}\|_{L^{4}(B(R))} \lessapprox R^{(2 - 2s)/4}, \qquad R \geq 1. \end{equation}
If $0 < s < 1$, and $p > 4$, then there exists a constant $\epsilon = \epsilon(p,s) > 0$ such that
\begin{displaymath} \|\hat{\mu}\|_{L^{p}(B(R))} \leq C_{p,s} R^{(2 - (2s + \epsilon))/p}, \qquad R \geq 1. \end{displaymath}
The function $s \mapsto \epsilon(p,s)$ is bounded away from $0$ on any compact subset of $(0,1)$.
\end{thm}

The inequality \eqref{form89} means that for every $\eta > 0$, there exists a constant $C_{\eta,s} > 0$ such that $\|\hat{\mu}\|_{L^{4}(B(R))} \leq C_{\eta,s} R^{(2 - 2s)/4 + \eta}$ for all $R \geq 1$. The exponent $2 - 2s$ in \eqref{form89} is sharp, and this part of the theorem is not hard to prove. In fact, the classical \emph{$L^{4}$-orthogonality method} immediately reduces the oscillatory problem to a Kakeya type problem regarding families of wave packets arising from the $s$-dimensional measure $\mu$. Identifying the sharp non-concentration condition for such wave packet families takes some work, see Lemma \ref{lemma1}. After this has been accomplished, however, the Kakeya problem can be solved with the standard $L^{2}$ method, see the proof of Theorem \ref{L4thm}.

The second part of Theorem \ref{main} is more complicated, and has the following corollary regarding the dimension of triple sums of fractal subsets of $\mathbb{P}$:
\begin{cor}\label{mainCor} For every $0 < s < 1$, there exists $\epsilon = \epsilon(s) > 0$ such that the following holds. Let $K \subset \mathbb{P}$ be a Borel set with $\Hd K = s$. Then,
\begin{displaymath} \Hd (K + K + K) \geq 2s + \epsilon. \end{displaymath}
\end{cor}
For two summands instead of three, the lower bound $\Hd (K + K) \geq 2s$ is sharp and easy (even easier than the first part of Theorem \ref{main}). The main point in Corollary \ref{mainCor} is the $\epsilon$-improvement over this "trivial" bound. 

\subsection{Connection to discrete problems} Corollary \ref{mainCor} is a "continuous" version of the following discrete question: if $P \subset \mathbb{P}$ is a finite set, how large is (at least) the cardinality of $|P + P + P|$, or more generally $|kP|$ for $k \geq 3$? A variant of the problem asks for upper bounds for the $n^{th}$ \emph{additive energy}
\begin{displaymath} E_{n}(P) = |\{(x_{1},\ldots,x_{k},y_{1},\ldots,y_{k}) \in P^{2n} : x_{1} + \ldots + x_{n} = y_{1} + \ldots + y_{n}\}|. \end{displaymath}
This formulation is the discrete analogue of the problem studied in Theorem \ref{main}. The problems are related by the inequality $E_{n}(P)|nP| \geq |P|^{2n}$, an easy consequence of Cauchy-Schwarz. Bourgain and Demeter \cite{MR3374964} showed that $E_{3}(P) \lesssim_{\epsilon} |P|^{7/2 + \epsilon}$ for all $\epsilon > 0$, and asked (\cite[Question 2.13]{MR3374964}) if the estimate can be improved to $E_{3}(P) \lesssim |P|^{3 + \epsilon}$ for all $\epsilon > 0$. The positive result yields $|P + P + P| \gtrsim_{\epsilon} |P|^{5/2 - \epsilon}$, and a positive answer to the question would yield the optimal result $|P + P + P| \gtrsim_{\epsilon} |P|^{3 - \epsilon}$. If $\delta \in (0,1]$, and $P \subset \mathbb{P}$ is assumed to be $\delta$-separated, then the optimal bound follows from the sharp $\ell^{2}$-decoupling theorem: $E_{3}(P) \lesssim_{\epsilon} \delta^{-\epsilon}|P|^{3}$, see \cite[Theorem 13.21]{MR3971577}.

For $E_{3}(P)$ and $|P + P + P|$, these are the best current results, as far as I know. However, Mudgal \cite[Corollary 1.2]{2020arXiv200809247M} has recently obtained improvement for higher energies: $E_{k}(P) \lesssim_{k} |P|^{2k - 3 + \epsilon(k)}$, where $\epsilon(k) = (1/4 - 1/7246) \cdot 2^{-k + 4}$ for $k \geq 4$. In particular, it follows that $|kP| \geq |P|^{3 - o_{k}(1)}$, where $o_{k}(1) \to 0$ as $k \to \infty$. 

In analogy, one might hope that $\lim_{n \to \infty} \Hd (nK) \geq 3\Hd K$, but this is clearly false if $\Hd K > 2/3$. A more plausible conjecture might be that $\lim_{n \to \infty} \Hd (nK) \geq \min\{3\Hd K,\Hd K + 1\}$. This is closely connected to the discussion in Section \ref{s:epsilonValue}. 

\subsection{Connection to Borel subrings and proof of Corollary \ref{mainCor}} The $\epsilon$-improvements in the second part of Theorem \ref{main} and Corollary \ref{mainCor} are closely connected with the Borel subring problem. This problem, solved independently by Edrgar-Miller \cite{MR1948103} and Bourgain \cite{MR1982147} around 2000, asked to show that every Borel subring $R \subset \R$ has $\Hd R \in \{0,1\}$. This follows immediately from Corollary \ref{mainCor}, as we will discuss in this section. The point here is not to announce a new solution (indeed the argument of Theorem \ref{main} relies on previous solutions), but rather to shed light on the problem of bounding $\|\hat{\mu}\|_{L^{p}}$. 

To deduce Corollary \ref{mainCor} from Theorem \ref{main}, we need a standard lemma:
\begin{lemma}\label{lemma9} Let $\mu$ be a non-trivial finite Borel measure on $\R^{d}$, and let $\mu_{\delta} := \mu \ast \psi_{\delta}$, $\delta > 0$, where $\psi_{\delta}(x) = \delta^{-d}\psi(x/\delta)$ is a standard approximate identity; $\psi \in C^{\infty}_{c}(\R^{d})$ with $\psi \geq 0$ and $\int \psi = 1$. Let $s \in [0,d]$, and assume that
\begin{displaymath} \|\mu_{\delta}\|_{L^{2}(\R^{d})}^{2} \leq \delta^{s - d}, \qquad 0 < \delta \leq \delta_{0}. \end{displaymath}
Then $|(\spt \mu)|_{\delta} \gtrsim_{\psi} \|\mu\|^{2} \cdot \delta^{-s}$ for all $0 < \delta \leq \delta_{0}$, and $\Hd (\spt \mu) \geq s$. Here $|\cdot|_{\delta}$ refers to the $\delta$-covering number, and $\Hd$ is Hausdorff dimension. \end{lemma}

\begin{proof} Note that $|\widehat{\psi_{c\delta}}(\xi)| \gtrsim 1$ for all $|\xi| \leq \delta^{-1}$ if $c > 0$ is sufficiently small (depending on the choice of $\psi$).Now, if $0 < \sigma < s$, then we have
\begin{displaymath} \int |\hat{\mu}(\xi)|^{2}|\xi|^{\sigma - d} \, d\xi \lesssim \sum_{j \geq 0} 2^{j(\sigma - d)} \int_{B(2^{j})} |\hat{\mu}(\xi)|^{2} \, d\xi \lesssim \sum_{j \geq 0} 2^{j(\sigma - d)} \int |\hat{\mu}(\xi)|^{2}|\widehat{\psi_{c2^{-j}}}(\xi)|^{2} \, d\xi. \end{displaymath} 
The integral on the right is $\|\mu \ast \psi_{c2^{-j}}\|_{2}^{2} \leq (c2^{-j})^{s - d}$, so the sum is finite for $\sigma < s$. It is well-known that this implies $\Hd (\spt \mu) \geq s$, see \cite[Theorem 8.7 \& Lemma 12.12]{zbMATH01249699}. The claim about $|(\spt \mu)|_{\delta}$ is even simpler, being based on the inequality $\|f\|_{L^{1}}^{2} \leq \mathrm{Leb}(\spt f)\|f\|_{L^{2}}^{2}$ applied to $f = \mu_{\delta}$. \end{proof}

Corollary \ref{mainCor} follows from Theorem \ref{main}, and Lemma \ref{lemma9}. Indeed, if $K \subset \mathbb{P}$ is Borel with $\Hd K = s \in (0,1)$, then for any $0 \leq \sigma < s$, Frostman's lemma (see \cite[Theorem 8.8]{zbMATH01249699}) yields a non-trivial measure $\mu$ with $\spt \mu \subset K$ and $\mu(B(x,r)) \leq r^{\sigma}$ for all $x \in \mathbb{P}$ and $r > 0$. Then 
\begin{displaymath} \|(\mu \ast \mu \ast \mu)_{\delta}\|_{L^{2}}^{2} = \|\widehat{\mu_{\delta}}\|_{L^{6}}^{6} \leq C_{\sigma} \delta^{2\sigma + \epsilon - 2}, \qquad \delta \in (0,1], \end{displaymath}
and consequently $\Hd (K + K + K) \geq \Hd \spt (\mu \ast \mu \ast \mu) \geq 2\sigma + \epsilon$. Since $\epsilon$ is bounded away from zero for $\sigma$ sufficiently close to $s$, Corollary \ref{mainCor} follows by letting $\sigma \to s$.

Next, let us see why Corollary \ref{mainCor} implies the non-existence of Borel subrings of intermediate dimension. Let $A \subset \R$ be any Borel set with $\Hd A = s$, where $s \in (0,1)$. Then $K := K_{A} := \{(t,t^{2}) : t \in \R\}$ is a Borel subset of $\mathbb{P}$ with $\Hd K = s$. Evidently $K \subset A \times A^{2}$, so 
\begin{equation}\label{form100} K + K + K \subset (A + A + A) \times (A^{2} + A^{2} + A^{2}). \end{equation}
If $A$ were a ring, then the right hand side would be contained in $A \times A$, and hence $\Hd (K + K + K) \leq \Hd (A \times A)$. We finally claim that that $\Hd (A \times A) = 2s$, which will contradict Corollary \ref{mainCor}. This follows from a folklore result on the dimension of orthogonal projections: if $K \subset \R^{2}$ is Borel, then $\Hd \{e \in S^{1} : \Hd \pi_{e}(K) < \tfrac{1}{2}\Hd K\} = 0$. For a proof, see \cite[Theorem 1.2]{MR2994685}. In particular, since $\Hd A > 0$, and $A$ is a ring, we have $s = \Hd A = \Hd (A + aA) \geq \tfrac{1}{2}\Hd(A \times A)$ for some $a \in A$.

We close this section by mentioning a related result of Raz and Zahl \cite[Theorem 1.14]{2021arXiv210807311R}. A special case of their theorem shows that if $A \subset [0,1]$ is a $(\delta,s,\delta^{-\epsilon})$-set with $s \in (0,1)$ (see Definition \ref{deltaSSet}) and $|A|_{\delta} \geq \delta^{-s + \epsilon}$, and $\epsilon = \epsilon(s) > 0$ is small enough, then 
\begin{displaymath} \max\{|A + A|_{\delta},|A^{2} + A^{2}|_{\delta}\} \geq \delta^{-s - \epsilon}. \end{displaymath}
The proof of Corollary \ref{mainCor} gives an alternative argument for this fact. Indeed, if $|A + A|_{\delta} \leq \delta^{-s - \epsilon} \leq \delta^{-2\epsilon}|A|_{\delta}$ and $|A^{2} + A^{2}|_{\delta} \leq \delta^{-s - \epsilon} \leq \delta^{-2\epsilon}|A|_{\delta}$, then it follows from the Pl\"unnecke's inequality (see \cite{MR266892} for the original reference, or \cite[Corollary 6.28]{MR2289012} for a textbook), that also 
\begin{displaymath} |A + A + A|_{\delta} \lesssim \delta^{-s - O(\epsilon)} \quad \text{and} \quad |A^{2} + A^{2} + A^{2}|_{\delta} \lesssim \delta^{-s - O(\epsilon)}. \end{displaymath}
Consequently, $|(A + A + A) \times (A^{2} + A^{2} + A^{2})|_{\delta} \lesssim \delta^{-2s - O(\epsilon)}$. Given the inclusion \eqref{form100}, this contradicts a $\delta$-discretised version of Corollary \ref{mainCor} for $\epsilon > 0$ small enough, depending only on $s \in (0,1)$ (the required $\delta$-discretised version follows from Theorem \ref{main}, using the easier part of Lemma \ref{lemma9}.) We note, however, that the theorem of Raz and Zahl concerns far more general non-linear images of $A \times A$ than just $A^{2} + A^{2}$.

\subsection{Value of $\epsilon$?}\label{s:epsilonValue} We do not know what the precise value of "$\epsilon$" should be for every pair $(p,s)$ with $s \in [0,1]$ and $p > 4$. However, the following conjecture seems plausible:
\begin{conjecture} For every $0 \leq s \leq 1$ and $\epsilon > 0$, there exists $p = p(\epsilon,s) \geq 1$ such that the following holds. Let $\mu$ be a Borel measure on $\mathbb{P}$ satisfying $\mu(B(x,r)) \leq r^{s}$ for all $x \in \R^{2}$ and $r > 0$. Then, 
\begin{equation}\label{form83} \|\hat{\mu}\|_{L^{p}(B(R))} \leq C_{\epsilon,s}R^{[2 - \min\{3s,1 + s\}]/p + \epsilon}, \qquad R \geq 1. \end{equation}
\end{conjecture}

It is not hard to see that the threshold $\min\{3s,1 + s\}$ cannot be further improved: 

\begin{ex} Consider a set $A \subset [0,1] \cap (\delta^{s}\Z)(\delta)$ which is a union of $\sim \delta^{-s}$ equally spaced intervals of length $\delta > 0$ in arithmetic progression. Then $A^{2} \subset [0,1] \cap (\delta^{2s}\Z)(\delta)$ can be covered by a union of $\sim \delta^{-2s}$ intervals with spacing $\delta^{2s}$, again in arithmetic progression. Therefore $kA^{2}$ can also be covered by $\lesssim_{k} \delta^{-2s}$ intervals with the same spacing $\delta^{2s}$. It follows that
\begin{equation}\label{form82} |kA \times kA^{2}|_{\delta} \lesssim_{k} \delta^{-3s}, \qquad k \geq 1. \end{equation}
If $s \geq \tfrac{1}{2}$, we can do better: then we note that $kA^{2}$ is trivially covered by $\lesssim_{k} \delta^{-1}$ intervals of length $\delta$, hence
\begin{equation}\label{form84} |kA \times kA^{2}|_{\delta} \lesssim_{k} \delta^{-1 - s}, \qquad k \geq 1. \end{equation}
The bounds \eqref{form82}-\eqref{form84} show (with the assistance of Lemma \ref{lemma9}, and a version of \eqref{form100} for $k$-fold sums) that the exponent $\min\{3s,1 + s\}$ in \eqref{form83} cannot be lowered.
\end{ex}

\subsection{Proof and paper outlines} Below Theorem \ref{main}, we already explained the key points needed to prove the first part of Theorem \ref{main}. The details are contained in Section \ref{s:L4}. The proof of the second part is based on the following ingredients:
\begin{itemize}
\item Knowing that the first part is true: the exponent "$2 - 2s$" will serve as a "base camp" from which we reach out for the $\epsilon$-improvement for $p > 4$.

\item An observation due to Bombieri, Bourgain, and Demeter, \cite{MR3373053,MR3374964}: the $3^{rd}$ additive energy of subsets of $\mathbb{P}$ is closely related to an incidence-counting problem between points in $\R^{2}$, and circles centred along the $x$-axis. A $\delta$-discretised version of this result is formulated at the beginning of Section \ref{s:epsilon}.
\item If the exponent "$2 - 2s$" were sharp for some $p > 4$, it turns out that we could construct an \emph{$(s,2s)$-Furstenberg set} of dimension $2s$. Furstenberg sets are introduced in Section \ref{s:furstenberg}. In particular, it is known (due to Bourgain, and H\'era-Shmerkin-Yavicoli) that $(s,2s)$-Furstenberg sets of dimension $2s$ do not exist (for $s \in (0,1)$). This is where the $\epsilon$-improvement in Theorem \ref{main} comes from.
\item Based on the hypothetical sharpness of the exponent "$2 - 2s$" for $p > 4$, and the relation between additive energies and circle incidences, we first construct an "$(s,2s)$-Furstenberg of circles". This is done in Section \ref{circleFurstenberg}. Luckily, the ensuing circles are all centred along the $x$-axis. It turns out that the incidence geometry of such circles is equivalent to the incidence geometry of planar lines. This is because there exists a well-behaved map between the \emph{Poincar\'e half-plane model} and the \emph{Beltrami-Klein model} of hyperbolic geometry. This was explained to me by Josh Zahl. The details of the transformation are contained in Section \ref{circlesToLines}, where the proof of Theorem \ref{main} is finally concluded.
\end{itemize}

\subsection*{Acknowledgements} The article was mostly written at the Hausdorff Research Institute for Mathematics, Bonn, during the trimester \emph{Interactions between Geometric measure theory, Singular integrals, and PDE.} I would like to thank the institute and its staff for their generous hospitality. I am grateful to Keith Rogers for dozens of discussions, and no fewer ideas, on the project during the trimester. Keith's input was so substantial that he should be viewed as a co-author, at least if the reader has a positive opinion of the paper (however, I take full credit for mistakes). I am grateful to Josh Zahl for telling me how to transform point-circle incidences to point-line incidences, and for providing references.

\section{Preliminaries}

\subsection{Notation} The notation $B(x,r)$ stands for a closed ball of radius $r > 0$ and centre $x \in X$, in a metric space $(X,d)$. If $A \subset X$ is a bounded set, and $r > 0$, we write $|A|_{r}$ for the $\delta$-covering number of $A$, that is, the smallest number of closed balls of radius $r$ required to cover $A$. Cardinality is denoted $|A|$, and Lebesgue measure $\mathrm{Leb}(A)$. The closed $r$-neighbourhood of $A$ is denoted $A(r)$. Throughout the article, $\mathbb{P}$ denotes the truncated parabola $\mathbb{P} = \{(t,t^{2}) : t \in [-1,1]\}$, and $S^{1} \subset \R^{2}$ is the unit circle. The notation "$\pi$" (without subindex) refers to the projection $\pi(x,y) = x$. 

 \subsection{Furstenberg sets and $(\delta,s)$-sets}\label{s:furstenberg} We have already seen that Theorem \ref{main} implies the non-existence of Borel subrings of intermediate dimension. This suggests that the non-existence proofs could be useful in establishing Theorem \ref{main}. This turns out to be the case, although not in a completely straightforward fashion. 

In fact, we will prove Theorem \ref{main} by applying the non-existence of $(s,2s)$-Furstenberg sets of dimension $2s$. The connection between the Furstenberg set problem and the Borel subring problem was initially discovered by Katz and Tao \cite{MR1856956}, and has been thereafter applied several times to make progress in the Furstenberg set problem, see \cite{2021arXiv211208249D,HSY21,Orponen20,2021arXiv210603338O}. In the current paper, we may use the best current estimates on $(s,2s)$-Furstenberg sets as a black box. We record the necessary preliminaries now. 

\begin{definition}[$(\delta,s,C)$-set]\label{deltaSSet} Let $(X,d)$ be a metric space, let $\delta \in (0,1]$, and $C,s > 0$. A finite set $P \subset X$ is called a \emph{$(\delta,s,C)$-set} if 
\begin{displaymath} |P \cap B(x,r)| \leq C\left(\frac{r}{\delta} \right)^{s}, \qquad x \in X, \, r \geq \delta. \end{displaymath}
\end{definition} 
If the value of the constant "$C$" is not relevant, we may also write "$(\delta,s)$-set" in place of "$(\delta,s,C)$-set". We will need this notion in $X = \R^{d}$, and also in the space of all affine lines in $\R^{2}$, denoted $\mathcal{A}(2,1)$. We denote by $d_{\mathcal{A}(2,1)}$ the following metric on $\mathcal{A}(2,1)$:
\begin{displaymath} d_{\mathcal{A}(2,1)}(\ell_{1},\ell_{2}) := \|\pi_{1} - \pi_{2}\| + |a_{1} - a_{2}|, \end{displaymath}
where $\pi_{1},\pi_{2}$ are the orthogonal projections to the subspaces $L_{1},L_{2}$ parallel to $\ell_{1},\ell_{2}$, $\|\cdot\|$ is the operator norm, and $a_{j}$ is the unique point on $\ell_{j} \cap L_{j}^{\perp}$. The only property of the metric $d_{\mathcal{A}(2,1)}$ we need explicitly is that if $\ell_{1},\ell_{2} \in \mathcal{A}(2,1)$ are at distance $r := d_{\mathcal{A}(2,1)}(\ell_{1},\ell_{2})$, then
\begin{equation}\label{form78} \ell_{1} \cap B(1) \subset \ell_{2}(Cr) \end{equation}
for some absolute constant $C > 0$. Implicitly, we also need that the results quoted below for Furstenberg sets are valid with the choice of metric $d_{\mathcal{A}(2,1)}$, but  this requirement is compatible with \eqref{form78}. 

\begin{definition}[$(\delta,t,C)$-set of lines] A set of lines $\mathcal{L} \subset \mathcal{A}(2,1)$ is a $(\delta,t,C)$-set if it is a $(\delta,s,C)$-set in the metric space $(\mathcal{A}(2,1),d_{\mathcal{A}(2,1)})$. \end{definition}

We may now define $\delta$-discretised Furstenberg sets:

\begin{definition}[Discretised $(s,t)$-Furstenberg set] Let $0 < s \leq 1$ and $0 < t \leq 2$, $C > 0$, and $\delta \in (0,1]$.  A set $F \subset \R^{2}$ is a $\delta$-discretised $(s,t,C)$-Furstenberg set if there exists a $(\delta,t,C)$-set of lines with $|\mathcal{L}| \geq C^{-1}\delta^{-t}$ with the property that $F \cap \ell(\delta)$ contains a $(\delta,s,C)$-set of cardinality $\geq C^{-1}\delta^{-s}$ for all $\ell \in \mathcal{L}$. \end{definition} 

It is known that if $0 < s < 1$, $s < t \leq 2$, and $\epsilon = \epsilon(s,t) > 0$ is small enough, then every $\delta$-discretised $(s,t,\delta^{-\epsilon})$-Furstenberg set $E \subset \R^{2}$ satisfies $|E|_{\delta} \geq \delta^{-2s - \epsilon}$. This is a recent result of the first author with Shmerkin \cite{2021arXiv210603338O}. In the present paper, we only need the special case $t = 2s$, which was known much earlier: the case $s = \tfrac{1}{2}$ is due to Bourgain \cite{MR1982147} (modulo a slightly different definition of discretised Furstenberg sets), and the general case $s \in (0,1)$ is due to H\'era, Shmerkin, and Yavicoli \cite{HSY21}. In the case $t = 2s$, the best known constant "$\epsilon$" is actually fairly large, due to recent work of Di Benedetto and Zahl \cite{2021arXiv211208249D}. We quote their version of the result below:

\begin{thm}\label{DBZ} Let $c(s) := s(1 - s)/(6(155 + 68s))$. Then, for every $0 < s < 1$ and every $c < c(s)$ there exists $\epsilon > 0$ such that the following holds for all $\delta > 0$ sufficiently small. Let $F \subset \R^{2}$ be a $\delta$-discretised $(s,2s,\delta^{-\epsilon})$-Furstenberg set. Then, $|F|_{\delta} \geq \delta^{-2s - c}$. \end{thm} 

All the results on Furstenberg sets mentioned above are based on reductions to the \emph{discretised sum-product theorem}: if $A \subset [0,1]$ is a $(\delta,s)$-set of cardinality $\sim \delta^{-s}$, $s \in (0,1)$, then either $|A + A|_{\delta} \geq \delta^{-s - \epsilon}$ or $|A \cdot A|_{\delta} \geq \delta^{-s - \epsilon}$ for some $\epsilon = \epsilon(s) > 0$. This result is originally due to Bourgain \cite{MR1982147} from the early 2000s, but a more quantitative version was proven recently by Guth, Katz, and Zahl \cite{MR4283564}. This result (combined with a more efficient reduction) enabled Di Benedetto and Zahl to prove Theorem \ref{DBZ}. We do not attempt to quantify the constant "$\epsilon$" appearing in Theorem \ref{main}. Theorem \ref{main} is based on a reduction to Theorem \ref{DBZ}, but not a particularly straightforward one. The value of "$\epsilon$" we obtain in Theorem \ref{main} is anyway much smaller than the constant "$c(s)$" in Theorem \ref{DBZ}.

\section{First part of Theorem \ref{main}}\label{s:L4}

In this section, we establish the $L^{4}$-bound in Theorem \ref{main}. The proof is based on the use of wave packet decompositions, which we briefly define in the next section.

\subsection{Wave packet decomposition} The material in this section is standard, and we follow Demeter's book \cite[Exercise 2.7]{MR3971577}. We use the notation $w(x) := (1 + |x|)^{-100}$. If $T \subset \R^{2}$ is a rectangle, we write $w_{T} := w \circ A_{T}$, where $A$ is the affine map taking $T$ to $[-1,1]^{2}$.

\begin{proposition}\label{WPProp} Let $\mu$ be a finite Borel measure with $\spt \mu \subset \mathbb{P}$, and let $\delta > 0$. Let $\psi_{\delta}(\xi) := \xi^{-2}\psi(\xi/\delta)$, where $\psi \in C^{\infty}_{c}(\R^{2})$ is a bump function with the properties $0 \leq \psi \leq \mathbf{1}_{B(1)}$ and $\widehat{\psi} \gtrsim \mathbf{1}_{B(1)}$. Write $\mu_{\delta} := \mu \ast \psi_{\delta}$. Then $\spt \mu_{\delta} \subset \mathbb{P}(\delta)$.

Let $\Theta$ be a finitely overlapping cover of $\mathbb{P}(\delta)$ by rectangles "$\theta$" of dimensions roughly $\delta \times \delta^{1/2}$ (also known as "\emph{caps}"). Let $\{\varphi_{\theta}\}_{\theta \in \Theta}$ be a smooth partition of unity adapted to the cover $\Theta$, and let $\mu_{\theta} := \mu_{\delta}\varphi_{\theta}$ for $\theta \in \Theta$. Then $\mu_{\delta} = \sum_{\theta \in \Theta} \mu_{\theta}$, and consequently $\widehat{\mu_{\delta}} = \sum_{\theta \in \Theta} \widehat{\mu_{\theta}}$. 

Fix $\theta \in \Theta$. Let $\mathcal{T}_{\theta}$ be a tiling of $\R^{2}$ by rectangles "$T$" dual to $2\theta$, with dimensions roughly $\delta^{-1} \times \delta^{-1/2}$. Then, for each $T \in \mathcal{T}_{\theta}$ we can associate a function $W_{T} \in \mathcal{S}(\R^{2})$ with the following properties:
\begin{itemize}
\item[(W1)] $\spt \widehat{W_{T}} \subset 2\theta$, $\|W_{T}\|_{L^{2}} \sim 1$, and 
\begin{displaymath} |W_{T}| \lesssim_{N} \mathrm{Leb}(T)^{-1/2} \cdot (w_{T})^{N}, \qquad N \geq 1. \end{displaymath}
\item[(W2)] If $\{a_{T}\}_{T \in \mathcal{T}_{\theta}}$ is an arbitrary collection of complex numbers, then
\begin{displaymath} \Big\| \sum_{T \in \mathcal{T}_{\theta}} a_{T}W_{T}\Big\|_{L^{2}}^{2} \sim \sum_{T \in \mathcal{T}_{\theta}} |a_{T}|^{2}. \end{displaymath}
\item[(W3)] If $F \in L^{2}(\R^{2})$ with $\spt F \subset \theta$, then
\begin{displaymath} \widehat{F} = \sum_{T \in \mathcal{T}_{\theta}} \langle \widehat{F},W_{T} \rangle W_{T}, \end{displaymath}
where $\langle \cdot,\cdot \rangle$ refers to inner product in $L^{2}$. 
\end{itemize}
In particular, it follows from property \textup{(W3)}, and $\spt \mu_{\theta} \subset \theta$, that 
\begin{equation}\label{form90} \widehat{\mu_{\delta}} = \sum_{\theta \in \Theta} \widehat{\mu_{\theta}} = \sum_{\theta \in \Theta} \sum_{T \in \mathcal{T}_{\theta}} a_{T}W_{T}, \qquad a_{T} = \langle \widehat{\mu_{\theta}},W_{T} \rangle \text{ for } T \in \mathcal{T}_{\theta}. \end{equation}
The representation \eqref{form90} is known as the \emph{wave packet decomposition of $\mu$ at scale $\delta$}.
\end{proposition}

\subsection{A non-concentration estimate for wave packets}

The following lemma shows that if $\mu$ is an $s$-dimensional Frostman measure on $\mathbb{P}$, then the wave packet decomposition of $\mu$ at scale $\delta \in (0,1]$ satisfies a $1$-dimensional non-concentration condition, regardless of $s \in [0,1]$. A more precise statement is \eqref{form91}: the "$1$-dimensionality" is visible in the exponent $\Delta = \Delta^{1}$. Why $1$-dimensional, and not $s$-dimensional? The heuristic (and imprecise) reason is the following. Since the measure $\mu$ is $s$-dimensional, the directions of the wave packets indeed satisfy an $s$-dimensional non-concentration condition. However, since $\mu_{\theta}$ is $s$-dimensional for each $\theta \in \Theta$, the wave packets associated to a fixed $\theta \in \Theta$ satisfy a $(1 - s)$-dimensional non-concentration condition (see \eqref{form92}). Finally, $1 = s + (1 - s)$.


\begin{lemma}\label{lemma1} Let $s \in [0,1]$, $\delta \in (0,1]$, and let $\mu$ be a Borel measure on parabola $\mathbb{P}$ satisfying $\mu(B(x,r)) \leq Cr^{s}$ for all $x \in \mathbb{P}$ and $r > 0$. Let $\delta > 0$, and let $\{a_{T}\}_{T \in \mathcal{T}_{\theta}}$, $\theta \in \Theta$, be the coefficients in the wave packet decomposition of $\mu$ at scale $\delta$. Fix $\epsilon > 0$, and let $S \subset \R^{2}$ be any rectangle with dimensions $\Delta \times R^{1 + \epsilon}$, where $R = \delta^{-1}$, and $R^{1/2} \leq \Delta \leq R^{1 + \epsilon}$. Then,
\begin{equation}\label{form91} \sum_{T \subset S} |a_{T}|^{2} \lesssim_{\epsilon} \Delta \cdot \delta^{s - 1 - 4\epsilon}. \end{equation}
\end{lemma}

\begin{proof} Fix a rectangle $S \subset \R^{2}$ as in the statement, and let $\Theta_{S} \subset \Theta$ be the caps with the property that $T \subset S$ for at least one rectangle $T \in \mathcal{T}_{\theta}$. Since the rectangles $T \in \mathcal{T}_{\theta}$ have longer side of length $R$, every $\theta \in \Theta_{S}$ lies inside a subset $J(\delta) \subset \mathbb{P}(\delta)$ of diameter $\diam(J(\delta)) \lesssim \Delta/R = \delta \Delta$ by elementary geometry. Therefore, 
\begin{displaymath} \sum_{T \subset S} |a_{T}|^{2} = \sum_{\theta \subset J(\delta)} \mathop{\sum_{T \in \mathcal{T}_{\theta}}}_{T \subset S} |a_{T}|^{2}. \end{displaymath}
To proceed, fix $\theta \in J(\delta)$. We create a bit of separation between the rectangles $T \in \mathcal{T}_{\theta}$ with the following trick. Partition $\mathcal{T}_{\theta}$ into $N \sim R^{\epsilon} = \delta^{-\epsilon}$ collections $\mathcal{T}_{\theta,j}$, $1 \leq j \leq N$, such that $\dist(T,T') \gtrsim R^{\epsilon}$ for all distinct $T,T' \in \mathcal{T}_{\theta,j}$ (for $1 \leq j \leq N$ fixed). We claim that
\begin{equation}\label{form92} \mathop{\sum_{T \in \mathcal{T}_{\theta,j}}}_{T \subset S} |a_{T}|^{2} \lesssim_{\epsilon} \Delta^{1 - s} \cdot \delta^{-1 - 3\epsilon}\mu(2\theta), \qquad 1 \leq j \leq N. \end{equation}
Once \eqref{form92} has been established, the proof of \eqref{form91} is concluded by summing over $1 \leq j \leq N$ and $\theta \in J(\delta)$, and finally using the $s$-Frostman estimate $\mu(2J(\delta)) \lesssim (\delta \Delta)^{s}$:
\begin{displaymath} \sum_{T \subset S} |a_{T}|^{2} \lesssim_{\epsilon} \Delta^{1 - s} \cdot \delta^{-1 - 4\epsilon} \sum_{\theta \in J(\delta)} \mu(2\theta) \lesssim \Delta^{1 - s} \cdot \delta^{-1 - 4\epsilon} \mu(2J(\delta)) \lesssim \Delta \cdot \delta^{s -1 - 4\epsilon}. \end{displaymath} 
 To prove \eqref{form92}, let $\mathcal{S}$ be a rectangle which is concentric with $S$, but inflated by a factor $R^{\epsilon}$ in both directions. Thus $\mathcal{S}$ is a rectangle of dimensions $R^{\epsilon}\Delta \times R^{1 + 2\epsilon}$. Then, let $\eta_{\mathcal{S}} \in C^{\infty}_{c}(\R^{2})$ be a bump function satisfying $\mathbf{1}_{\mathcal{S}} \leq \eta_{\mathcal{S}} \leq \mathbf{1}_{2\mathcal{S}}$. We observe that if $T \subset S$, then
 \begin{displaymath} |a_{T}| = |\langle \widehat{\mu_{\theta}},W_{T} \rangle| \leq |\langle \widehat{\mu_{\theta}}\eta_{\mathcal{S}},W_{T} \rangle| + \|W_{T}\|_{L^{1}(\R^{2} \, \setminus \, \mathcal{S})} \lesssim_{\epsilon}  |\langle \widehat{\mu_{\theta}}\eta_{\mathcal{S}},W_{T} \rangle| + \delta, \end{displaymath}
 using that $\|\widehat{\mu_{\theta}}\|_{L^{\infty}} \leq \|\mu_{\theta}\|_{L^{1}} \leq 1$, and the rapid decay $|W_{T}| \lesssim_{\epsilon} (w_{T})^{1/\epsilon}$ outside $T \subset S$, stated in Proposition \ref{WPProp}(W1). Taking further into account the (very crude) estimate $|\{T \in \mathcal{T}_{\theta} : T \subset S\}| \lesssim R = \delta^{-1}$, we find that
\begin{equation}\label{form94} \mathop{\sum_{T \in \mathcal{T}_{\theta,j}}}_{T \subset S} |a_{T}|^{2} \lesssim_{\epsilon} \mathop{\sum_{T \in \mathcal{T}_{\theta,j}}}_{T \subset S} |\langle \widehat{\mu_{\theta}}\eta_{\mathcal{S}},W_{T} \rangle|^{2} + \delta. \end{equation}
To estimate the the main term in \eqref{form94}, we apply the abstract inequality
\begin{displaymath} \sum |\langle x,e_{k} \rangle|^{2} \leq \|x\|^{2} + \sum_{k \neq l} \langle x,e_{k} \rangle\langle x,e_{l} \rangle \langle e_{k},e_{l} \rangle, \end{displaymath}
valid for all inner product spaces $(H,\langle \cdot,\cdot \rangle)$, all $x \in H$, and all finite sequences of unit vectors $\{e_{k}\}_{k} \subset H$. (Proof: expand the inequality $0 \leq \|x - \sum \langle x,e_{k} \rangle e_{k} \|^{2}$.) Applying this to $x = \widehat{\mu_{\theta}}\eta_{\mathcal{S}} \in L^{2}$, and $e_{T} := W_{T}/\|W_{T}\|_{2}$, and recalling that $\|W_{T}\|_{2} \sim 1$, the result is
\begin{equation}\label{form93} \mathop{\sum_{T \in \mathcal{T}_{\theta,j}}}_{T \subset S} |\langle \widehat{\mu_{\theta}}\eta_{\mathcal{S}},W_{T} \rangle|^{2} \lesssim \|\widehat{\mu_{\theta}}\eta_{\mathcal{S}}\|_{L^{2}}^{2} + \mathop{\sum_{T,T' \in \mathcal{T}_{\theta,j}}}_{T,T' \subset S, \, T \neq T'} |\langle \widehat{\mu_{\theta}}\eta_{S},W_{T} \rangle\langle \widehat{\mu_{\theta}}\eta_{S},W_{T'} \rangle \langle W_{T},W_{T'} \rangle|. \end{equation} 
The number of terms in the second sum is $\lesssim R^{2}$, and also $|\langle \widehat{\mu_{\theta}}\eta_{\mathcal{S}},W_{T} \rangle| \lesssim \|W_{T}\|_{L^{1}} \lesssim R^{2}$.  These factors are negligible compared to the fact that $|\langle W_{T},W_{T'} \rangle| \lesssim_{\epsilon} R^{-10}$, which follows from $\dist(T,T') \geq R^{\epsilon}$ for distinct $T,T' \in \mathcal{T}_{\theta,j}$, and the rapid decay of $W_{T}$ outside $T$. Combining \eqref{form93} with \eqref{form94}, and using Plancherel, we find that
\begin{displaymath} \mathop{\sum_{T \in \mathcal{T}_{\theta,j}}}_{T \subset S} |a_{T}|^{2} \lesssim_{\epsilon} \|\mu_{\theta} \ast \widehat{\eta_{\mathcal{S}}}\|_{L^{2}}^{2} + \delta. \end{displaymath}
To estimate $\|\mu_{\theta} \ast \widehat{\eta_{\mathcal{S}}}\|_{L^{2}}^{2}$, we first record that $\widehat{\eta_{\mathcal{S}}}$ is essentially supported in the dual rectangle $\mathcal{S}^{\ast}$ of $\mathcal{S}$, which has dimensions $\delta^{\epsilon}\Delta^{-1} \times \delta^{1 + 2\epsilon}$. In particular, $\mathcal{S}^{\ast}$ fits inside the ball $B(\Delta^{-1})$ with room to spare. From this, we first deduce that
\begin{displaymath} |\widehat{\eta_{\mathcal{S}}}| \lesssim_{\epsilon} \mathrm{Leb}(\mathcal{S}) \cdot w_{B(\Delta^{-1})} \sim \Delta \cdot R^{1 + 3\epsilon} \cdot w_{B(\Delta^{-1})}. \end{displaymath}
Since $\|\mu_{\theta} \ast w_{B(\Delta^{-1})}\|_{L^{\infty}} \lesssim \Delta^{-s}$ by the $s$-Frostman assumption of $\mu$, we arrive at
\begin{displaymath} \|\mu_{\theta} \ast \widehat{\eta_{\mathcal{S}}}\|_{L^{\infty}} \lesssim \Delta \cdot R^{1 + 3\epsilon} \|\mu_{\theta} \ast w_{B(\Delta^{-1})}\|_{L^{\infty}} \lesssim \Delta^{1 - s} \cdot R^{1 + 3\epsilon}, \end{displaymath} 
and finally
\begin{displaymath} \|\mu_{\theta} \ast \widehat{\eta_{\mathcal{S}}}\|_{L^{2}}^{2} \lesssim \|\widehat{\eta_{\mathcal{S}}} \ast \mu_{\theta}\|_{L^{\infty}} \|\widehat{\eta_{\mathcal{S}}}\|_{L^{1}}\|\mu_{\theta}\|_{L^{1}} \lesssim \Delta^{1 - s} \cdot R^{1 + 3\epsilon}\mu(2\theta). \end{displaymath}
This concludes the proof of \eqref{form92}, and the lemma. \end{proof}

\subsection{Proof of the $L^{4}$-estimate}

In this section, we complete the proof of the first part of Theorem \ref{main}.

\begin{thm}\label{L4thm} Let $0 \leq s \leq 1$, and let $\mu$ be a Borel measure on $\Gamma = \{(t,t^{2}) : t \in [-1,1]\}$ satisfying $\mu(B(x,r)) \leq r^{s}$ for all $x \in \mathbb{P}$ and $r > 0$. Then, for every $\epsilon > 0$, there exists $C = C_{\epsilon,s} > 0$ such that 
\begin{displaymath} \|\hat{\mu}\|_{L^{4}(B(R))} \leq CR^{(1 - s)/2 + \epsilon}, \qquad R \geq 1. \end{displaymath} 
\end{thm}

\begin{proof} Fix $R \geq 1$, write $\delta := R^{-1}$. Let $\psi_{\delta}$ be the approximate identity appearing in the wave packet decomposition, Proposition \ref{WPProp}, so $\widehat{\psi_{\delta}} \gtrsim \mathbf{1}_{B(R)}$. Therefore, $\|\hat{\mu}\|_{L^{4}(B(R))}^{4} \lesssim \|\widehat{\mu_{\delta}}\|_{L^{4}(\R^{2})}^{4}$, where $\mu_{\delta} = \mu \ast \psi_{\delta}$. We then expand $\widehat{\mu_{\delta}}$ as in \eqref{form90}:
\begin{displaymath} \widehat{\mu_{\delta}} = \sum_{\theta \in \Theta} \widehat{\mu_{\theta}} = \sum_{\theta \in \Theta} \sum_{T \in \mathcal{T}_{\theta}} a_{T}W_{T}.\end{displaymath}
We record at this point that, by Proposition \ref{WPProp}(W2), we have
\begin{equation}\label{form96} \sum_{\theta \in \Theta} \sum_{T \in \mathcal{T}_{\theta}} |a_{T}|^{2} \sim \sum_{\theta \in \Theta} \|\widehat{\mu_{\theta}}\|_{L^{2}}^{2} \lesssim \|\mu_{\delta}\|_{L^{2}}^{2} \leq \|\mu_{\delta}\|_{L^{\infty}}\|\mu_{\delta}\|_{L^{1}} \lesssim \delta^{s - 2}, \end{equation}
recalling that $\mu_{\delta} = \mu \ast \psi_{\delta}$, where $\|\psi_{\delta}\|_{L^{\infty}} \lesssim \delta^{-2}$.

All the "oscillation" in our problem can be removed by an appeal to the C\'ordoba-Fefferman $L^{4}$-orthogonality lemma, see \cite{MR688026,MR320624} for original references. The form we need is recorded, for example, in \cite[Proposition 3.3]{MR3971577}:
\begin{equation}\label{form99} \|\widehat{\mu_{\delta}}\|_{L^{4}(\R^{2})}^{4} \lesssim \int \Big( \sum_{\theta \in \Theta} |\widehat{\mu_{\theta}}|^{2} \Big)^{2} = \sum_{\theta,\theta'} \int \Big| \sum_{T \in \mathcal{T}_{\theta}} a_{T}W_{T} \Big|^{2} \Big| \sum_{T' \in \mathcal{T}_{\theta'}} a_{T'}W_{T'} \Big|^{2}.  \end{equation}
If the supports of the functions $W_{T}$, $T \in \mathcal{T}_{\theta}$ were disjoint, we could expand the right hand side as
\begin{equation}\label{form97} \sum_{T,T' \in \mathcal{T}} |a_{T}|^{2}|a_{T'}|^{2} \int |W_{T}W_{T'}|^{2}, \end{equation} 
where $\mathcal{T}$ stands for the union of all the families $\mathcal{T}_{\theta}$, $\theta \in \Theta$. This is not quite accurate, and we resort to a trick we already employed in the proof of Lemma \ref{lemma1}, namely splitting the collections $\mathcal{T}_{\theta}$ into $N \sim R^{\epsilon}$ sub-collections $\mathcal{T}_{\theta,j}$ where the tubes $T \in \mathcal{T}_{\theta,j}$ are separated by at least $\gtrsim R^{\epsilon}$. Using the trivial inequality $|c_{1} + \ldots + c_{N}|^{2} \leq N^{2}\max |c_{j}|^{2}$, we first estimate
\begin{equation}\label{form98} \int \Big| \sum_{T \in \mathcal{T}_{\theta}} a_{T}W_{T} \Big|^{2} \Big| \sum_{T' \in \mathcal{T}_{\theta'}} a_{T'}W_{T'} \Big|^{2} \lesssim R^{4\epsilon} \max_{i,j} \int \Big| \sum_{T \in \mathcal{T}_{\theta,i}} a_{T}W_{T} \Big|^{2} \Big| \sum_{T' \in \mathcal{T}_{\theta',j}} a_{T'}W_{T'} \Big|^{2}. \end{equation}
Now, using the rapid decay of the functions $W_{T}$ outside $T$, and the (crude) uniform bound $|a_{T}| \leq \|W_{T}\|_{L^{1}} \lesssim R$, the integrands satisfy the following pointwise bounds:
\begin{displaymath} \Big| \sum_{T \in \mathcal{T}_{\theta,i}} a_{T}W_{T} \Big|^{2} \lesssim_{\epsilon} \sum_{T \in \mathcal{T}_{\theta,i}} |a_{T}|^{2}|W_{T}|^{2} + \delta \quad \text{and} \quad \Big| \sum_{T' \in \mathcal{T}_{\theta',j}} a_{T'}W_{T'} \Big|^{2} \lesssim_{\epsilon} \sum_{T' \in \mathcal{T'}_{\theta',j}} |a_{T'}|^{2}|W_{T'}|^{2} + \delta. \end{displaymath} 
When these bounds are plugged back into \eqref{form98}, and then \eqref{form99}, we finally arrive at the following replacement of \eqref{form97}:
\begin{equation}\label{form95} \|\widehat{\mu_{\delta}}\|_{L^{4}(\R^{2})}^{4} \lesssim \int \Big( \sum_{\theta \in \Theta} |\widehat{\mu_{\theta}}|^{2} \Big)^{2} \lesssim_{\epsilon} R^{4\epsilon} \sum_{T,T' \in \mathcal{T}} |a_{T}|^{2}|a_{T'}|^{2} \int |W_{T}W_{T'}|^{2} + \delta. \end{equation}
To estimate the right hand side, we imitate the proof of the $L^{2}$-Kakeya maximal function bound, with the only non-trivial addition of inserting Lemma \ref{lemma1} at a suitable point. We fix $T \in \mathcal{T}_{\theta_{0}} \subset \mathcal{T}$, and we decompose
\begin{equation}\label{form29} \sum_{T' \in \mathcal{T}} |a_{T'}|^{2} \int |W_{T}W_{T'}|^{2} = \sum_{\delta^{1/2} \lesssim \alpha \lesssim 1} \mathop{\sum_{T' \in \mathcal{T}}}_{\angle(T,T') \sim \alpha} |a_{T'}|^{2} \int |W_{T}W_{T'}|^{2}. \end{equation}
To be precise, the summation over $\delta^{1/2} \leq \alpha \leq 1$ runs over dyadic rationals in the indicated range, and the summation $\{T' \in \mathcal{T} : \angle(T,T') \sim \alpha\}$ runs over the tubes in those families $\mathcal{T}_{\theta}$ with $\alpha \leq |\theta - \theta_{0}| \leq 2\alpha$. Fix $\delta^{1/2} \lesssim \alpha \lesssim 1$. Fix also $T' \in \mathcal{T}$ with $\angle (T,T') \sim \alpha$, and let $\mathbf{T},\mathbf{T}'$ be tubes which are concentric with $T,T'$, but fattened by a factor $R^{\epsilon}$ in both directions. If $\mathbf{T} \cap \mathbf{T'} = \emptyset$, then 
\begin{displaymath} \int |W_{T}W_{T'}|^{2} \lesssim_{\epsilon} \delta^{10} \end{displaymath}
by the rapid decay of $W_{T},W_{T'}$. Therefore, the part of the sum \eqref{form29} over such $T' \in \mathcal{T}$ is bounded from above by $\lessapprox \delta^{10} \sum_{T' \in \mathcal{T}} |a_{T'}|^{2} \lesssim \delta^{8 + s} \leq \delta$, applying also \eqref{form96}. 

Assume then that $\mathbf{T} \cap \mathbf{T'} \neq \emptyset$. Recall that $\mathbf{T},\mathbf{T}'$ are rectangles of dimensions $R^{1/2 + \epsilon} \times R^{1 + \epsilon}$. Therefore $\mathrm{Leb}(\mathbf{T} \cap \mathbf{T}') \lesssim R^{1 + 2\epsilon}/\alpha$. Using this, and Proposition \ref{WPProp}(W1), we first deduce that
\begin{equation}\label{form28} \int |W_{T}W_{T'}|^{2} \lesssim_{\epsilon} \int_{\mathbf{T} \cap \mathbf{T'}} |W_{T}W_{T'}|^{2} + \delta \lesssim \mathrm{Leb}(T)^{-2} \cdot \frac{R^{1 + 2\epsilon}}{\alpha} = \frac{R^{2\epsilon - 2}}{\alpha}. \end{equation}
Moreover, since $\angle(T,T') = \angle(\mathbf{T},\mathbf{T}') \sim \alpha$, the non-empty intersection of $\mathbf{T},\mathbf{T}'$ implies that $T' \subset \mathbf{T}' \subset \mathbf{T}(\Delta) =: S$, where $\Delta \sim \alpha \cdot R^{1 + \epsilon} \geq R^{1/2}$. As usual, the notation stands $\mathbf{T}(\Delta)$ stands for the $\Delta$-neighbourhood of $\mathbf{T}$, which is a rectangle of dimensions roughly $\Delta \times R^{1 + \epsilon}$, noting that $\Delta \lesssim R^{1 + \epsilon}$. Consequently, applying Lemma \ref{lemma1}, we have
\begin{align*} \mathop{\sum_{T' \in \mathcal{T}}}_{\angle(T,T') \sim \alpha, \mathbf{T} \cap \mathbf{T}' \neq \emptyset} |a_{T'}|^{2} \int |W_{T}W_{T'}|^{2} & \stackrel{\eqref{form28}}{\lesssim_{\epsilon}} \frac{R^{2\epsilon - 2}}{\alpha} \cdot \mathop{\sum_{T \in \mathcal{T}}}_{T' \subset S} |a_{T'}|^{2}\\
& \stackrel{\textup{L. } \ref{lemma1}}{\lesssim_{\epsilon}}  \frac{R^{2\epsilon - 2}}{\alpha} \cdot \left( \Delta \cdot \delta^{s - 1 - 4\epsilon} \right) \sim R^{7\epsilon - s}.  \end{align*} 
Plugging this back into \eqref{form95}-\eqref{form29}, we see that
\begin{displaymath} \|\hat{\mu}\|_{L^{4}(B(R))}^{4} \lesssim_{\epsilon} R^{4\epsilon} \sum_{T \in \mathcal{T}} |a_{T}|^{2} \sum_{\delta^{1/2} \lesssim \alpha \lesssim 1} R^{7\epsilon - s} \lessapprox R^{11\epsilon - s} \sum_{T \in \mathcal{T}} |a_{T}|^{2} \stackrel{\eqref{form96}}{\lesssim} R^{2 - 2s + 11\epsilon}. \end{displaymath}
This completes the proof of the proposition. \end{proof}

\section{Second part of Theorem \ref{main}}\label{s:epsilon}

The purpose of this section is to prove the second part of Theorem \ref{main}, concerning exponents $p > 4$. In fact, since we already know the $p = 4$ endpoint from Theorem \ref{L4thm}, it suffices to establish the $\epsilon$-improvement in Theorem \ref{main} for $p = 6$. Namely, if this is already known, and $4 < p \leq 6$, then $p = 4\theta + 6(1 - \theta)$ for some $\theta < 1$, and hence
\begin{displaymath} \|\hat{\mu}\|_{L^{p}(B(R))}^{p} \leq \|\hat{\mu}\|_{L^{4}(B(R))}^{4\theta} \|\hat{\mu}\|_{L^{6}(B(R))}^{6(1 - \theta)} \stackrel{\textup{T. } \ref{L4thm}}{\lessapprox} R^{(2 - 2s)\theta} \cdot R^{(2 - 2s - \epsilon)(1 - \theta)} = R^{2 - 2s - \epsilon(1 - \theta)}. \end{displaymath} 
The cases $p > 6$ follow from the trivial estimate $\|\hat{\mu}\|_{L^{p}(B(R))}^{p} \lesssim_{p} \|\hat{\mu}\|_{L^{6}(B(R))}^{6}$, using only that $\|\hat{\mu}\|_{L^{\infty}} \leq \mu(\mathbb{P}) \leq 2$ for every measure as in the hypothesis of Theorem \ref{main}. We then restate the case $p = 6$ of Theorem \ref{main}:

\begin{thm}\label{thm2} For every $s \in (0,1)$, there exist $C = C(s) > 0$ and $\epsilon = \epsilon(s) > 0$ such that the following holds. Let $\mu$ be a Borel measure on $\mathbb{P}$ satisfying $\mu(B(x,r)) \leq r^{s}$ for all $x \in \mathbb{P}$ and $r > 0$. Then, 
\begin{equation}\label{form38} \|\hat{\mu}\|_{L^{6}(B(R))}^{6} \leq CR^{2 - 2s - \epsilon}, \qquad R \geq 1. \end{equation}
\end{thm}

\subsection{Auxiliary results} The $L^{6}(B(R))$-norm of $\hat{\mu}$ equals the $L^{2}$-norm of the convolution $\mu \ast \mu \ast \mu$, roughly speaking mollified at scale $R^{-1}$. This quantity, on the other hand, counts $R^{-1}$-approximate solutions to the equation $x_{1} + x_{2} + x_{3} = x_{4} + x_{5} + x_{6}$, with $(x_{1},\ldots,x_{6}) \in \spt \mu =: P$ (see Lemma \ref{lemma3}). It is well-known that if $P \subset \mathbb{P}$ is a finite set, then the problem of counting such (exact) solutions is connected to an incidence-counting problem for circles in the plane. The connection was discovered by Bourgain and Bombieri \cite{MR3373053} (for $P \subset S^{1}$) and then Bourgain and Demeter \cite{MR3374964} (for $P \subset \mathbb{P}$). The connection is captured by the following lemma. The case $\delta = 0$ is sketched in \cite[Proposition 2.15]{MR3374964}, but we give the details (we anyway need the details to prove the "approximate" version):

\begin{lemma}\label{demeterLemma} Let $\xi_{1},\xi_{2},\xi_{3} \in \R$, and write
\begin{equation}\label{form30} (\xi_{1},\xi_{1}^{2}) + (\xi_{2},\xi_{2}^{2}) + (\xi_{3},\xi_{3}^{2}) =: (a,b). \end{equation}
Then, the point
\begin{displaymath} A(\xi_{1},\xi_{2}) := (3(\xi_{1} + \xi_{2}),\sqrt{3}(\xi_{1} - \xi_{2})) \end{displaymath}
is contained on the circle $S_{a,b} := \partial B((2a,0),\sqrt{6b - 2a^{2}})$. 

The following approximate version also holds. Assume that $\delta \in (0,1]$, $\xi_{1},\xi_{2},\xi_{3} \in [-1,1]$, and \eqref{form30} is replaced by
\begin{equation}\label{form70} |(\xi_{1},\xi_{1}^{2}) + (\xi_{2},\xi_{2}^{2}) + (\xi_{3},\xi_{3}^{2}) - (a,b)| \leq \delta. \end{equation}
If $r := 6b - 2a^{2} \leq \delta$, then $A(\xi_{1},\xi_{2}) \in B((2a,0),C\sqrt{\delta})$ for an absolute constant $C > 0$. Otherwise, $A(\xi_{1},\xi_{2})$ is contained in the annulus $S_{a,b}(C\delta/\sqrt{r})$.
\end{lemma}

\begin{remark}\label{rem1} The proof below also shows that $6b - 2a^{2} \geq 0$, whenever $(a,b)$ arises as in \eqref{form30}. If the conclusion of Lemma \ref{demeterLemma} seems unintuitive at first, the following remark might be helpful: by \eqref{form30}, the point $(\xi_{1},\xi_{2},\xi_{3}) \in \R^{3}$ lies on the circle in $\R^{3}$ obtained by  intersecting the sphere $\xi_{1}^{2} + \xi_{2}^{2} + \xi_{3}^{2} = b$ with the plane $\xi_{1} + \xi_{2} + \xi_{3} = a$.  The fact that $A(\xi_{1},\xi_{2})$ lies on the planar circle $S_{a,b}$ could be "derived" from this observation with some effort, but in the following proof it is simpler to just "check" the conclusion.   \end{remark}

\begin{proof}[Proof of Lemma \ref{demeterLemma}] We first record that the equation \eqref{form30} yields
\begin{equation}\label{form69} (a - \xi_{3})^{2} \stackrel{\eqref{form30}}{=} (\xi_{1} + \xi_{2})^{2} = \xi_{1}^{2} + \xi_{2}^{2} + 2\xi_{1}\xi_{2} \stackrel{\eqref{form30}}{=} b - \xi_{3}^{2} + 2\xi_{1}\xi_{2},  \end{equation}
or in other words
\begin{equation}\label{form31} 2\xi_{1}\xi_{2} = a^{2} - 2a\xi_{3} + 2\xi_{3}^{2} - b. \end{equation}
After this observation, the rest of the argument is rather straightforward. To check that $A(\xi_{1},\xi_{2}) \in S((2a,0),\sqrt{6b - 2a^{2}})$, we simply calculate the distance $|A(\xi_{1},\xi_{2}) - (2a,0)|^{2} = $
\begin{align*} (3(\xi_{1} + \xi_{2}) - 2a)^{2} + 3(\xi_{1} - \xi_{2})^{2} & \stackrel{\eqref{form30}}{=} (a - 3\xi_{3})^{2} + 3(\xi_{1}^{2} + \xi_{2}^{2} - 2\xi_{1}\xi_{2})\\
& \stackrel{\eqref{form30}}{=} (a^{2} - 6a\xi_{3} + 9\xi_{3}^{2}) + 3(b - \xi_{3}^{2} - 2\xi_{1}\xi_{2})\\
& \stackrel{\eqref{form31}}{=} (a^{2} - 6a\xi_{3} + 9\xi_{3}^{2}) + 3(b - \xi_{3} - a^{2} + 2a\xi_{3} - 2\xi_{3}^{2} + b)\\
& = 6b - 2a^{2}.  \end{align*}
This is what we claimed in the first part of the statement. 

The second part follows by inspecting the calculation above. Using $\max\{|\xi_{1}|,|\xi_{2}|,|\xi_{3}|\} \leq 1$, the calculation \eqref{form69} combined with \eqref{form70} shows that
\begin{displaymath} 2\xi_{1}\xi_{2} = a^{2} - 2a\xi_{3} + 2\xi_{3}^{2} - b + O(\delta). \end{displaymath} 
This leads to 
\begin{displaymath} (3(\xi_{1} + \xi_{2}) - 2a)^{2} + 3(\xi_{1} - \xi_{2})^{2} = 6b - 2a^{2} + O(\delta). \end{displaymath}
In case $r = 6b - 2a^{2} \leq \delta$, we may conclude that $(3(\xi_{1} + \xi_{2}),\sqrt{3}(\xi_{1} - \xi_{2})) \in B((2a,0),C\sqrt{\delta})$. In the opposite case we use $|\sqrt{r} - \sqrt{s}| \leq |r - s|/\sqrt{r}$ to estimate
\begin{displaymath} |\sqrt{(3(\xi_{1} + \xi_{2}) - 2a)^{2} + 3(\xi_{1} - \xi_{2})^{2}} - \sqrt{6b - 2a^{2}}| \lesssim \delta/\sqrt{r}, \end{displaymath} 
so $A(\xi_{1},\xi_{2}) \in S_{a,b}(C\delta/\sqrt{r})$ as claimed. (The latter estimates are also valid if $0 < r < \delta$, but in this case the bound $\delta/\sqrt{r}$ is not very useful.) \end{proof}

We next formalise the connection of Fourier transforms and additive energies:

\begin{lemma}\label{lemma3} Let $P_{1},\ldots,P_{6} \subset \R^{d}$ be $\delta$-separated sets, and let $\mu_{1},\ldots,\mu_{6} \in C^{\infty}_{c}(\R^{d})$ be functions satisfying $0 \leq \mu \leq \mathbf{1}_{P_{j}(\delta)}$. Then,
\begin{displaymath} \int \widehat{\mu_{1}}\widehat{\mu_{2}}\widehat{\mu_{3}}\overline{\widehat{\mu_{4}}\widehat{\mu_{5}}\widehat{\mu_{6}}} \lesssim \delta^{5d}|\{(x_{1},\ldots,x_{6}) \in P_{1} \times \cdots \times P_{6} : |(x_{1} + x_{2} + x_{3}) - (x_{4} + x_{5} + x_{6})| \leq 6\delta\}|. \end{displaymath} 
\end{lemma}

\begin{proof} By Plancherel,
\begin{displaymath} \int \widehat{\mu_{1}}\widehat{\mu_{2}}\widehat{\mu_{3}}\overline{\widehat{\mu_{4}}\widehat{\mu_{5}}\widehat{\mu_{6}}} = \int (\mu_{1} \ast \mu_{2} \ast \mu_{3})(\mu_{4} \ast \mu_{5} \ast \mu_{6}). \end{displaymath} 
For $r > 0$, write
\begin{displaymath} m(z) := |\{(x_{1},x_{2},x_{3}) \in P_{1} \times P_{2} \times P_{3} : |(x_{1} + x_{2} + x_{3}) - z| \leq r\}| \end{displaymath}
and 
\begin{displaymath} n(z) := |\{(x_{4},x_{5},x_{6}) \in P_{4} \times P_{5} \times P_{6} : |(x_{4} + x_{5} + x_{6}) - z| \leq r\}|. \end{displaymath}
Then,
\begin{align*} (\mu_{1} \ast \mu_{2} \ast \mu_{3})(z) & = \iint \mu_{1}(z - x_{2} - x_{3})\mu_{2}(x_{2})\mu_{3}(x_{3}) \, dx_{2} \, dx_{3}\\
& \lesssim \delta^{2d} \sum_{(x_{2},x_{3}) \in P_{2} \times P_{3}} \mathbf{1}_{P_{1}(3\delta)}(z - x_{2} - x_{3})\\
& \leq \delta^{2d} \sum_{x_{1} \in P_{1}} |\{(x_{2},x_{3}) \in P_{2} \times P_{3} : |z - (x_{1} + x_{2} + x_{3})| \leq 3\delta\}| = \delta^{2d}m_{3\delta}(z),  \end{align*} 
and similarly $(\mu_{4} \ast \mu_{5} \ast \mu_{6})(z) \lesssim \delta^{2d}n_{3\delta}(z)$. Therefore,
\begin{align*} \int & (\mu_{1} \ast \mu_{2} \ast \mu_{3})(\mu_{4} \ast \mu_{5} \ast \mu_{6}) \lesssim \delta^{4d} \int m_{3\delta}(z)n_{3\delta}(z) \, dz\\
& = \delta^{4d} \sum_{x_{1},\ldots,x_{6}} \mathrm{Leb}(\{z \in \R^{d} : |z - (x_{1} + x_{2} + x_{3})| \leq 3\delta \text{ and } |z - (x_{4} + x_{5} + x_{6})| \leq 3\delta\}). \end{align*}
The sum runs over $(x_{1},\ldots,x_{6}) \in P_{1} \times \cdots \times P_{6}$, and it can evidently be restricted to those $6$-tuples with $|(x_{1} + x_{2} + x_{3}) - (x_{4} + x_{5} + x_{6})| \leq 6\delta$. For such $6$-tuples, on the other hand, the possible $z \in \R^{d}$ lie in a ball of radius $\sim \delta$, and their Lebesgue measure contributes the $5^{th}$ factor of "$\delta^{d}$". This completes the proof of the lemma. \end{proof}

Finally, we record the following consequence of transversality:
\begin{lemma}\label{lemma4} Let $0 < \delta \leq \tau \leq 1$. Let $I,J \subset \mathbb{P}$ be arcs with $\dist(I,J) \geq \tau$, and let $P_{I} \subset I$ and $P_{J} \subset J$ be $\delta$-separated sets. Then, for all $x_{0},y_{0} \in \R^{2}$, and $C > 0$, we have
\begin{displaymath} |\{(x,y) \in P_{I} \times P_{J} : |(x + x_{0}) \pm (y + y_{0})| \leq C\delta\}| \lesssim C^{2}/\tau. \end{displaymath} 
\end{lemma}

\begin{proof} First, we record that
\begin{equation}\label{form101} \diam((x_{0} + I)(C\delta) \cap \pm(y_{0} + J)(C\delta)) \lesssim C\delta/\tau. \end{equation} 
This follows by parametrising the arcs $x_{0} + I =: G(f_{I})$ and $\pm(y_{0} + J) =: G(f_{J})$ as graphs of some quadratic functions "$f_{I}$" and "$f_{J}$", and noting that $(f_{I} - f_{J})' \gtrsim \tau$ (or $(f_{J} - f_{I})' \gtrsim \tau$) by assumption. We extend $x_{0} + I$ and $y_{0} + J$ so that $f,g \in C^{1}(\R)$, $f,g$ are $2$-Lipschitz, and the inequality $(f_{I} - f_{J})'(x) \gtrsim \tau$ remains valid for all $x \in \R$. Clearly $\diam (\{x \in \R : |(f - g)(x)| \leq r\}) \lesssim r/\tau$. Finally, it follows from the fact that $f,g$ are $2$-Lipschitz that every point $(x,y) \in G(f_{I})(C\delta) \cap G(f_{J})(C\delta)$ has $|(f - g)(x)| \lesssim C\delta$. This gives \eqref{form101}.

Now, if $(x,y) \in P_{I} \times P_{J}$ satisfies $|(x + x_{0}) \pm (y + y_{0})| \leq C\delta$, then certainly 
\begin{displaymath} x \in -x_{0} + \left( (x_{0} + I)(C\delta) \cap \pm(y_{0} + J)(C\delta) \right). \end{displaymath}
Since $P_{I}$ is $\delta$-separated, it follows from \eqref{form101} that the number of admissible "$x$" is $\lesssim C/\tau$. Finally, for every admissible $x \in P_{I}$, the number possible $y \in P_{J}$ satisfying $|(x + x_{0}) \pm (y + y_{0})| \leq C\delta$ is $\lesssim C$, by the $\delta$-separation of $P_{J}$. This completes the proof.  \end{proof}

\subsection{Proof of Theorem \ref{thm2}: initial reductions} Let $s \in (0,1)$, and let $\mu$ be a measure as in Theorem \ref{thm2}, satisfying $\mu(B(x,r)) \leq r^{s}$ for all $x \in \mathbb{P}$ and $r > 0$. In this section, the implicit constants in the "$\lesssim$" notation are allowed to depend on "$s$". We claim that
\begin{equation}\label{form44} \int |\hat{\mu}(\xi)|^{6} \chi_{R}(\xi) \, d\xi \lesssim R^{2 - 2s - \epsilon}, \qquad R \geq 1, \end{equation}
for some $\epsilon = \epsilon(s) > 0$, where $\chi_{R} \in \mathcal{S}(\R^{2})$ a Schwartz function satisfying $\chi_{R} \gtrsim \mathbf{1}_{B(R)}$, decaying rapidly outside $B(2R)$, and with $\spt \widehat{\chi_{R}} \subset B(\delta)$ (as usual $\delta = R^{-1}$). Concretely, it will be useful to take $\chi_{R}$ of the form
\begin{equation}\label{form102} \chi_{R} = \left( \widehat{\varphi_{\delta}} \right)^{6}, \end{equation}
where $\varphi_{\delta}(x) = \delta^{-2}\varphi(x/\delta)$, and $\varphi \in C^{\infty}(B(1))$, and $\widehat{\varphi_{\delta}} \geq 0$.

Here is a brief and informal description of the proof. We fix a small parameter $\epsilon = \epsilon(s) > 0$. We will first reduce the proof of \eqref{form44} to an "extremal" situation where the measure $\mu$ is concentrated on $\lesssim \delta^{-s - \epsilon}$ arcs $I \subset \mathbb{P}$ of length $\delta$, each satisfying $\delta^{s + \epsilon} \lesssim \mu(I) \leq \delta^{s}$. Roughly speaking, if the measure $\mu$ fails to look like this, the estimate \eqref{form44} will readily follow from Theorem \ref{L4thm}. After this reduction, in the next section, we will make the counter assumption that \eqref{form44} fails for some measure of the kind mentioned above. This information is then used to construct a $\delta$-discretised $(s,2s,\delta^{-C\epsilon})$-Furstenberg set $F \subset \R^{2}$ with $|F|_{\delta} \lesssim \delta^{-2s - C\epsilon}$, for some absolute constant $C > 0$. Choosing $\epsilon > 0$ sufficiently small will finally violate Theorem \ref{DBZ}, and the proof of \eqref{form44} will be complete.

We turn to the details. We start by reducing the proof of \eqref{form44} to the case where $\mu$ is "essentially constant" at scale $\delta$. This is a simple consequence of pigeonholing, but let us make the statement precise. Given a dyadic rational $r \in 2^{-\N}$, let $\mathcal{D}_{r}$ be the partition of $[-1,1)$ to dyadic intervals of length $r$. For $I \in \mathcal{D}_{r}$, we also write $\tilde{I} \subset \mathbb{P}$ for the arc "above" $I$ on $\mathbb{P}$. We define $\mu_{I} := \mu|_{\tilde{I}}$. Now, we claim that in order to prove \eqref{form44}, it suffices to do so for measures $\mu$ with the following extra property: there exists a constant $\kappa \in 2^{-\N}$ such that if $I \in \mathcal{D}_{\delta}$, then either
\begin{equation}\label{form59} \mu_{I} \equiv 0 \quad \text{or} \quad \kappa \leq \mu_{I}(\mathbb{P}) = \mu(\tilde{I}) \leq 2\kappa. \end{equation}
To see this, note that every measure $\mu$, as in the statement of the theorem, can be written as a series $\mu = \sum_{\kappa \in 2^{-\N}} \mu_{\kappa}$, where $\mu_{\kappa}$ has the additional property \eqref{form59}. Moreover, since $\|\widehat{\mu_{\kappa}}\|_{\infty} \leq \mu_{\kappa}(\R^{2}) \lesssim \kappa \cdot |\mathcal{D}_{\delta}| \sim \kappa R$, we have the trivial estimate
\begin{displaymath} \int \Big| \sum_{\kappa \leq \delta^{2}} \widehat{\mu_{\kappa}}(\xi) \Big|^{6} \chi_{R}(\xi) \, d\xi \lesssim \sum_{\kappa \leq \delta^{2}} \kappa^{-1} \int |\widehat{\mu_{\kappa}}(\xi)|^{6} \chi_{R}(\xi) \, d\xi \lesssim R^{8} \sum_{\kappa \leq \delta^{2}} \kappa^{5} \lesssim R^{-2}. \end{displaymath}
This is much better than what we claim at \eqref{form44}. On the other hand, since the sum over $\delta^{2} \leq \kappa \leq 1$ only contains $\lesssim \log(1/\delta) = \log R$ terms, we also have
\begin{displaymath} \int \Big| \sum_{\kappa \geq \delta^{2}} \widehat{\mu_{\kappa}}(\xi) \Big|^{6} \chi_{R}(\xi) \, d\xi \lesssim (\log R) \cdot R^{2 - 2s - \epsilon}, \end{displaymath} 
assuming that \eqref{form44} has already been established for each measure $\mu_{\kappa}$ individually. Thus, \eqref{form44} holds with "$\epsilon/2$" instead of "$\epsilon$" for the original measure $\mu$. 

From now on, we assume that $\mu$ satisfies the additional property \eqref{form59} for some $\kappa \in 2^{-\N}$. Another simple initial reduction is this: we may assume that
\begin{equation}\label{form66} \mu(\R^{2}) \geq \delta^{\epsilon} \end{equation}
for a small constant $\epsilon = \epsilon(s) > 0$ (whose value will be determined during the proof). Indeed, in the opposite case $\|\hat{\mu}\|_{L^{\infty}} \leq \delta^{\epsilon}$, and
\begin{displaymath} \int |\hat{\mu}(\xi)|^{6} \chi_{R}(\xi) \, d\xi \leq \delta^{2\epsilon} \int |\hat{\mu}(\xi)|^{4} \chi_{R}(\xi) \, d\xi \lesssim_{\epsilon} R^{2 - 2s - \epsilon}, \end{displaymath}
by Theorem \ref{L4thm} (or rather a version of it with the smooth cut-off $\chi_{R}$, which is easy to deduce from the proper statement).

We next reduce the proof of \eqref{form44} to a (partially) bilinear statement. To this end, let $\mathcal{W}$ be a Whitney decomposition of the set $\Omega := [-1,1)^{2} \, \setminus \, \{(x,x) : x \in [-1,1)\}$ into squares of the form $Q = I \times J$, where $I,J \in \mathcal{D}_{r}$ for some $r \in 2^{-\N}$. With this notation, we may write
\begin{align*} \hat{\mu}(\xi)^{2} = \iint e^{-2\pi i(x + y)\xi} \, d\mu(x) \, d\mu(y) & \stackrel{(\ast)}{=} \sum_{I \times J \in \mathcal{W}} \iint e^{-2\pi i(x + y)\xi} \, d\mu_{I}(x) \, d\mu_{J}(y)\\
& = \sum_{I \times J \in \mathcal{W}} \widehat{\mu_{I}}(\xi)\widehat{\mu_{J}}(\xi). \end{align*} 
Recall here that $\mu_{I} := \mu|_{\tilde{I}}$, where $\tilde{I} \subset \mathbb{P}$ is the arc "above" $I \in \mathcal{D}_{r}$. The equation $(\ast)$ uses the fact that the "Whitney squares" $\tilde{I} \times \tilde{J}$ partition $\mu \times \mu$ almost all of $\mathbb{P} \times \mathbb{P}$. For future reference, we immediately record the estimates
\begin{equation}\label{form56} \|\widehat{\mu_{I}}\|_{\infty} \leq \mu(\tilde{I}) \lesssim \ell(I)^{s} \quad \text{and} \quad \|\widehat{\mu_{J}}\|_{\infty} \leq \mu(\tilde{J}) \lesssim \ell(J)^{s}. \end{equation}
Now, we decompose
\begin{align*} \int |\hat{\mu}(\xi)|^{6} \chi_{R}(\xi) \, d\xi & = \int |\hat{\mu}(\xi)|^{2}\hat{\mu}(\xi)^{2}\overline{\hat{\mu}(\xi)}^{2} \chi_{R}(\xi) \, d\xi\\
& = \sum_{I \times J \in \mathcal{W}} \int |\hat{\mu}(\xi)|^{2}\widehat{\mu_{I}}(\xi)\widehat{\mu_{J}}(\xi)\overline{\hat{\mu}(\xi)}^{2} \chi_{R}(\xi) \, d\xi. \end{align*}
We denote the individual terms on the right hand side $\mathcal{F}(I \times J)$. To estimate these terms, fix a "separation constant" of the form $\tau := \delta^{\epsilon/100}$. Then, we write
\begin{displaymath} \int |\hat{\mu}(\xi)|^{6} \chi_{R}(\xi) \, d\xi = \mathop{\sum_{I \times J \in \mathcal{W}}}_{\ell(I) = \ell(J) \leq \tau} \mathcal{F}(I \times J) + \mathop{\sum_{I \times J \in \mathcal{W}}}_{\ell(I) = \ell(J) > \tau} \mathcal{F}(I \times J) =: \mathcal{F}_{\leq \tau} + \mathcal{F}_{> \tau}. \end{displaymath}
The main work of the proof will be to show that $\mathcal{F}_{> \tau} \leq R^{2 - 2s - \epsilon}$ if $\epsilon = \epsilon(s) > 0$ is chosen sufficiently small. A much easier task, carried out immediately below, is to show that $\mathcal{F}_{\leq \tau} \lessapprox R^{2 - 2s - \epsilon s/200}$. To do this, fix $I \times J \in \mathcal{W}$ with $\ell(I) = \ell(J) < \tau$, and start by applying \eqref{form56} and then Theorem \ref{L4thm} to estimate
\begin{displaymath} \mathcal{F}(I \times J) \lesssim \mu(\tilde{I})\ell(J)^{s} \int |\hat{\mu}(\xi)|^{4} \chi_{R}(\xi) \, d\xi \lesssim_{\epsilon} \mu(\tilde{I})\ell(J)^{s} R^{2 - 2s + \epsilon s/200}. \end{displaymath} 
By the properties of Whitney squares in the domain $\Omega$, if $I \times J \in \mathcal{W}$, then $J \subset CI$ for some absolute constant $C > 0$. This allows us to estimate as follows:
\begin{align*} \mathcal{F}_{\leq \tau} & \lesssim_{\epsilon} R^{2 - 2s + \epsilon s/200} \mathop{\sum_{I \times J \in \mathcal{W}}}_{\ell(I) = \ell(J) \leq \tau} \mu(\tilde{I})\ell(J)^{s}\\
& \leq R^{2 - 2s + \epsilon s/200} \sum_{r \leq \tau} r^{s} \sum_{I \in \mathcal{D}_{r}} \mathop{\sum_{J \in \mathcal{D}_{r}}}_{J \subset CI} \mu(\tilde{I}) \lesssim R^{2 - 2s + \epsilon s/200} \tau^{s} = R^{2 - 2s - \epsilon s/200}. \end{align*} 
This is what we claimed regarding the term $\mathcal{F}_{\leq \tau}$, so in the sequel we focus on $\mathcal{F}_{> \tau}$. We note that the number of elements in $\{I \times J \in \mathcal{W} : \ell(I) = \ell(I) \geq \tau\}$ is $\lesssim R^{\epsilon/50}$. It now suffices to prove an upper bound of the following form for the individual terms in the definition of $\mathcal{F}_{> \tau}$:
\begin{equation}\label{form57} \mathcal{F}(I \times J) \lesssim R^{2 - 2s - \epsilon}. \end{equation}
Once this has been accomplished, we may deduce that
\begin{displaymath} \mathcal{F}_{> \tau} \lesssim R^{2 - 2s - \epsilon} \cdot |\{I \times J \in \mathcal{W} : \ell(I) = \ell(J) \geq \tau\}| \leq R^{2 - 2s - \epsilon/2}. \end{displaymath}
This will conclude the proof of Theorem \ref{thm2}.

Most of the proof of \eqref{form57} will be contained in the next sections, but here we still reduce it to a special case where the constant "$\kappa$" from \eqref{form59} satisfies $\kappa \geq \delta^{s + \epsilon}$ (the upper bound $\kappa \lesssim \delta^{s}$ is also true, and follows immediately from the $s$-Frostman condition of $\mu$). To this end, fix $I \times J \in \mathcal{W}$ with $\ell(I) = \ell(J) \geq \tau$. Start by expanding
\begin{equation}\label{form60} \mathcal{F}(I \times J) = \int \hat{\mu}\widehat{\mu_{I}}\widehat{\mu_{J}}\overline{\hat{\mu}\hat{\mu}\hat{\mu}} \cdot \chi_{R} \stackrel{\eqref{form102}}{=} \int \hat{\mu}\widehat{\mu_{I}}\widehat{\mu_{J}}\overline{\hat{\mu}\hat{\mu}\hat{\mu}} \cdot \widehat{\varphi_{\delta}}^{6}, \end{equation} 
and recalling that $\varphi_{\delta} = \delta^{-2}\varphi(\cdot/\delta) \in C^{\infty}_{c}(\R^{2})$ satisfies $\spt \varphi_{\delta} \subset B(\delta)$. Since $\mu(B(x,\delta)) \lesssim \kappa$ for all $x \in \mathbb{P}$ by \eqref{form59}, we have
\begin{equation}\label{form58} \|\mu \ast \varphi_{\delta}\|_{\infty} \lesssim \delta^{-2}\kappa, \end{equation}
and $\mu_{I},\mu_{J}$ satisfy a similar estimate, being restrictions of $\mu$. Now, Lemma \ref{lemma3} will be applicable to the right hand side of \eqref{form60}. To make this precise, let $P,P_{I},P_{J}$ be $\delta$-nets in the supports of $\mu,\mu_{I},\mu_{J}$, respectively. Taking into account \eqref{form58}, Lemma \ref{lemma3} (with $d = 2$) applied to the functions $\mu_{j} \in \{\mu \ast \varphi_{\delta},\mu_{I} \ast \varphi_{\delta},\mu_{J} \ast \varphi_{\delta}\}$, $1 \leq j \leq 6$, implies that
\begin{displaymath} \mathcal{F}(I \times J) \lesssim \delta^{-2}\kappa^{6} |\{(x_{1},\ldots,x_{6}) \in P_{I} \times P_{J} \times P^{4}: |(x_{1} + x_{2} + x_{3}) - (x_{4} + x_{5} + x_{6})| \leq 6\delta\}|. \end{displaymath}
We expand the count over the $6$-tuples as
\begin{displaymath} \sum_{x_{3},\ldots,x_{6} \in P} |\{(x_{1},x_{2}) \in P_{I} \times P_{J} : |(x_{1} + x_{2} + x_{3}) - (x_{4} + x_{5} + x_{6})| \leq 6\delta\}|.  \end{displaymath} 
Since $\dist(I,J) \geq \tau$, it follows from Lemma \ref{lemma4} that each term in the sum here is $\lesssim \tau^{-1} = \delta^{-\epsilon/100}$. Consequently,
\begin{equation}\label{form62} \mathcal{F}(I \times J) \lesssim \delta^{-2 - \epsilon/100}\kappa^{6} \cdot |P|^{4} \lesssim \delta^{-2 - \epsilon/100}\kappa^{2}. \end{equation}
In the second inequality, we used the lower bound $\mu(I) \geq \kappa$ from the almost constancy property \eqref{form59} to deduce that $|P| \lesssim \kappa^{-1}$. From the inequality above, we finally see that if $\kappa \leq \delta^{s + \epsilon}$, then $\mathcal{F}(I \times J) \lesssim \delta^{2s - 2 - \epsilon/100 + 2\epsilon} \leq R^{2 - 2s - \epsilon}$, and \eqref{form57} has been established. So, the remaining -- and most substantial -- case in the proof of \eqref{form57} is where $\kappa \geq \delta^{s + \epsilon}$. In this case, we record that
\begin{equation}\label{form46} |P| \lesssim \delta^{-s - \epsilon}. \end{equation}
where we recall that $P$ is a $\delta$-net in $\spt \mu$. We record here that $P$ is a $(\delta,s,C\delta^{-2\epsilon})$-set of cardinality $|P| \gtrsim \delta^{-s + \epsilon}$. Indeed, if $x \in \mathbb{P}$ and $r \geq \delta$, note that $B(x,2r)$ contains a $\delta$-arc $\tilde{I}$ of $\mu$ measure $\mu(\tilde{I}) \sim \kappa \geq \delta^{s + \epsilon}$ around every point $y \in P \cap B(x,r)$ (this is because of the $\kappa$-almost constancy property of $\mu$, and $P \subset \spt \mu$). Therefore,
\begin{displaymath} |P \cap B(x,r)| \lesssim \kappa^{-1}\mu(B(x,2r)) \lesssim \delta^{-\epsilon} \cdot \left(\frac{r}{\delta} \right)^{s}, \qquad x \in \mathbb{P}, \, r \geq \delta. \end{displaymath} 
The lower bound $|P| \gtrsim \delta^{-s + \epsilon}$ follows from \eqref{form66}: indeed $\delta^{\epsilon} \leq \mu(\R^{2}) \lesssim |P|\kappa \lesssim |P|\delta^{s}$.

\subsection{Finding an $(s,2s)$-Furstenberg set of circles}\label{circleFurstenberg} We then proceed to prove the inequality \eqref{form57} under the assumption \eqref{form46}. In brief, we will show that if \eqref{form46} fails, then we can construct a "$2s$-dimensional" family of circles centred along the $x$-axis, all of which contain an "$s$-dimensional" subset of a fixed "$2s$-dimensional set". This will eventually lead to a contradiction with the non-existence of $2s$-dimensional $(s,2s)$-Furstenberg sets.

We have already seen, as a consequence of Lemma \ref{lemma3}, that
\begin{displaymath} \mathcal{F}(I \times J) \lesssim \delta^{6s-2} |\{(x_{1},\ldots,x_{6}) \in P_{I} \times P_{J} \times P^{4}: |(x_{1} + x_{2} + x_{3}) - (x_{4} + x_{5} + x_{6})| \leq 8\delta\}|, \end{displaymath}
where we already plugged in the (trivial) upped bound $\kappa \lesssim \delta^{s}$. Let $E_{3}$ be the cardinality of $6$-tuples on the right hand side. What remains to be done is to show that
\begin{equation}\label{form61} E_{3} \leq \delta^{-4s + \epsilon} \end{equation} 
for some $\epsilon = \epsilon(s) > 0$, and for all $\delta > 0$ small enough. This will be true if (i) the separation $\dist(I,J) \geq \delta^{\epsilon/100}$ is valid for $\epsilon > 0$ small enough, and (ii) the upper bound \eqref{form46} holds for $\epsilon > 0$ small enough, both requirements only depending on $s \in (0,1)$. 
 To prove \eqref{form61}, we start by expanding
\begin{equation}\label{form67} E_{3} = \sum_{x_{3},y_{1},y_{2},y_{3} \in P} |\{(x_{1},x_{2}) \in P_{I} \times P_{J} : |(x_{1} + x_{2} + x_{3}) - (y_{1} + y_{2} + y_{3})| \leq 6\delta\}|. \end{equation}
By the separation $\dist(P_{I},P_{J}) \geq \tau = \delta^{\epsilon/100} \geq \delta^{\epsilon}$, and Lemma \ref{lemma4}, we have the uniform upper bound
\begin{equation}\label{form50} |\{(x_{1},x_{2}) \in P_{I} \times P_{J} : |x_{1} + x_{2} + X| \leq 6\delta\}| \lesssim \delta^{-\epsilon}, \quad X \in \R^{2}. \end{equation}
We now make the counter assumption that 
\begin{equation}\label{form51} E_{3} \gtrapprox \delta^{-4s}. \end{equation} 
Here, and in the remainder of the argument, the notation "$\lessapprox$" and "$\gtrapprox$" is allowed to hide constants of the form $C\delta^{-C\epsilon}$ for an absolute constant $C > 0$. So, in particular \eqref{form50} tells us that the left hand side is $\lessapprox 1$ for all $X \in \R^{2}$. We also say that a set $P' \subset \R^{d}$ is a $(\delta,t)$-set if $P'$ is a $(\delta,t,C)$-set with $C \lessapprox 1$. 

Now, apply \eqref{form50} to $X = x_{3} - (y_{1} + y_{2} + y_{3})$, as in \eqref{form67}. Recall that $E_{3} \gtrapprox \delta^{-4s}$ by \eqref{form51}, and on the other hand the sum in \eqref{form67} only contains $\lessapprox \delta^{-4s}$ terms, by \eqref{form46}. These facts together imply that
\begin{equation}\label{form52} |\{(x_{1},x_{2}) \in P_{I} \times P_{J} : |(x_{1} + x_{2} + x_{3}) - (y_{1} + y_{2} + y_{3})| \leq 6\delta\}| \geq 1 \end{equation}
for $\gtrapprox \delta^{-4s}$ quadruples $(x_{3},y_{1},y_{2},y_{3}) \in P^{4}$.

We restate this information in more convenient form. Since the number of quadruples $(x_{3},y_{1},y_{2},y_{3})$ satisfying \eqref{form52} is $\gtrapprox \delta^{-4s}$, and $|P| \lessapprox \delta^{-s}$, there exists a fixed point $Y := y_{3} \in P$ such that \eqref{form52} holds for $\gtrapprox \delta^{-3s}$ triples $(x_{3},y_{1},y_{2}) \in P^{3}$. This point $Y \in P$ will not change during the remainder of the proof. 

Further, since the number of triples $(x_{3},y_{1},y_{2})$ is $\gtrapprox \delta^{-3s}$, and again $|P| \lessapprox \delta^{-s}$, we may deduce the following: there exists a set $\mathcal{S} \subset P \times P$ of $|\mathcal{S}| \approx \delta^{-2s}$ pairs $(y_{1},y_{2})$ with the property that \eqref{form52} holds for $\gtrapprox \delta^{-s}$ different choices $x_{3} \in P$ (for $y_{1},y_{2},Y$ fixed). In symbols, the cardinality of the set
\begin{equation}\label{form55} P(y_{1},y_{2}) := \{x_{3} \in P : \eqref{form52} \text{ holds for the quadruple $(x_{3},y_{1},y_{2},Y)$}\} \end{equation} 
is $|P(y_{1},y_{2})| \gtrapprox \delta^{-s}$ for all $(y_{1},y_{2}) \in \mathcal{S}$.

We briefly explain what happens next before giving the details. The set $\mathcal{S}$ will be identified with a "$(\delta,2s)$-set of circles" $S_{(y_{1},y_{2})} \subset \R^{2}$, all centred along the $x$-axis. Given a circle $S = S_{(y_{1},y_{2})}$ with $(y_{1},y_{2}) \in \mathcal{S}$, the condition $|P(y_{1},y_{2})| \gtrapprox \delta^{-s}$ will translate into the statement that the $(\approx \delta)$-neighbourhood of $S$ contains a $(\delta,s)$-set of cardinality $\gtrapprox \delta^{-s}$. Finally, it turns out that the union of all these $(\delta,s)$-sets is contained in a set of the form $F := T(P \times P)$, where $T \colon \R^{4} \to \R^{2}$ is an $O(1)$-Lipschitz linear map. In particular, $|F|_{\delta} \lesssim |P \times P| \lessapprox \delta^{-2s}$. These properties allow us to build (in Section \ref{circlesToLines}) a $\delta$-discretised $(s,2s)$-Furstenberg set of cardinality $\lessapprox \delta^{-2s}$, and this will violate Theorem \ref{DBZ}. 

We then define the sets $\mathcal{S}$ and $F$, which are inspired by Lemma \ref{demeterLemma}. For every $(y_{1},y_{2}) \in \mathcal{S}$ (or more generally $(y_{1},y_{2}) \in \R^{2} \times \R^{2}$), we write
\begin{equation}\label{form54} y_{1} + y_{2} + Y =: \sigma = (\sigma_{1},\sigma_{2}) \in \R^{2}, \end{equation}
and we define the circle
\begin{equation}\label{form63} S_{(y_{1},y_{2})} := S_{\sigma_{1},\sigma_{2}} = \partial B\left((2\sigma_{1},0),\sqrt{6\sigma_{2} - 2\sigma_{1}^{2}}\right). \end{equation} 
The notation $S_{\sigma_{1},\sigma_{2}}$ is familiar from Lemma \ref{demeterLemma}, and $6\sigma_{2} - 2\sigma_{1}^{2} \geq 0$, as observed in Remark \ref{rem1}. The definition of $S_{(y_{1},y_{2})}$ also depends on $Y \in P$, but this point can be viewed as "fixed forever". We then define the ("Furstenberg") set $F$ as
\begin{equation}\label{defF} F := \{(3(\pi(x) + \pi(y)),\sqrt{3}(\pi(x) - \pi(y))) : x,y \in P\}, \end{equation}
where as usual $\pi(\xi_{1},\xi_{2}) = \xi_{1}$. Evidently $F$ is the image of the $\delta$-separated set $P \times P$ under a certain $O(1)$-Lipschitz linear map $T \colon \R^{4} \to \R^{2}$. In particular,
\begin{equation}\label{form53} |F|_{\delta} \lesssim |P \times P| \lessapprox \delta^{-2s}. \end{equation}
The linear map $T$ is closely connected with the map "$A$" from Lemma \ref{demeterLemma}, indeed $T(x,y) = A(\pi(x),\pi(y))$ for all $(x,y) \in \R^{4}$.

Next, we claim that if $\mathbf{C} \approx 1$ is a suitable constant, and $(y_{1},y_{2}) \in \mathcal{S}$ is fixed, then there exists a $(\delta,s)$-set 
\begin{equation}\label{form103} F_{(y_{1},y_{2})} \subset F \cap S_{(y_{2},y_{2})}(\mathbf{C}\delta) \quad \text{with} \quad |F_{(y_{1},y_{2})}| \approx \delta^{-s}. \end{equation}
We will infer this by showing that the circle $S_{(y_{1},y_{2})}$ has radius $\approx 1$, and
\begin{equation}\label{form68} [\pi(F \cap S_{(y_{2},y_{2})}(\mathbf{C}\delta)](C\delta) \supset 3 \cdot [\pi(\sigma) - \pi(P(y_{1},y_{2}))], \end{equation}
where $\sigma = y_{1} + y_{2} + Y$, and $C > 0$ is an absolute constant. This implies \eqref{form103}. Indeed, recall that $P(y_{1},y_{2}) \subset \mathbb{P}$ is a subset of $P$ of cardinality $|P(y_{1},y_{2})| \approx \delta^{-s}$, for every $(y_{1},y_{2}) \in \mathcal{S}$. We observed below \eqref{form46} that $P$ is a $(\delta,s)$-set of cardinality $|P| \approx \delta^{-s}$, and these properties are inherited (with slightly worse constants) by any subset of cardinality $\approx \delta^{-s}$. In particular, $P(y_{1},y_{2})$ is a $(\delta,s)$-set. Since $P(y_{1},y_{2}) \subset \mathbb{P}$, the same is true of $\pi(P(y_{1},y_{2}))$, and therefore the set on the right hand side of \eqref{form68}. Thus, \eqref{form68} shows that $\pi(F \cap S_{(y_{2},y_{2})}(\mathbf{C}\delta))$ contains a $(\delta,s)$-set of cardinality $\gtrapprox \delta^{-s}$. But since $S_{(y_{2},y_{2})}$ has radius radius $\approx 1$ (as we will prove), it follows that $F \cap S_{(y_{2},y_{2})}(\mathbf{C}\delta)$ itself must contain a $(\delta,s)$-set $F_{(y_{1},y_{2})}$ of cardinality $|F_{(y_{1},y_{2})}| \approx \delta^{-s}$, as claimed.

We then verify the inclusion \eqref{form68}. Fix $(y_{1},y_{2}) \in \mathcal{S}$, and write $\sigma := y_{1} + y_{2} + Y$. Denote $r := 6\sigma_{2} - 2\sigma_{1}^{2}$, the square of the radius of the circle $S_{(y_{1},y_{2})} = S_{\sigma_{1},\sigma_{2}}$ defined in \eqref{form63}. 

Continuing with the proof of \eqref{form68}, we fix $x_{3} \in P(y_{1},y_{2})$. By definition of $P(y_{1},y_{2})$ (see \eqref{form55}), this means that the property \eqref{form52} holds for the quadruple $(x_{3},y_{1},y_{2},Y)$: there exists at least one pair $(x_{2},x_{3}) \in P_{I} \times P_{J}$ such that
\begin{equation}\label{form64} |(x_{1} + x_{2} + x_{3}) - \sigma| = |(x_{1} + x_{2} + x_{3}) - (y_{1} + y_{2} + Y)| \leq 6\delta. \end{equation}
We then apply Lemma \ref{demeterLemma}. Write $x_{j} =: (\xi_{j},\xi_{j}^{2})$ for $1 \leq j \leq 3$. It follows from \eqref{form64} and Lemma \ref{demeterLemma} that either
\begin{equation}\label{form65} T(x_{1},x_{2}) = (3(\pi(x_{1}) + \pi(x_{2})),\sqrt{3}(\pi(x_{1}) - \pi(x_{2}))) = A(\xi_{1},\xi_{2}) \in S_{(y_{1},y_{2})}(C\delta/\sqrt{r}), \end{equation}  
or
\begin{displaymath} T(x_{1},x_{2}) \in B((2\sigma_{1},0),C\sqrt{\delta}), \end{displaymath}
where the latter case occurs if $r \leq \delta$. In both cases $T(x_{1},x_{2}) \in F$, by definition (see \eqref{defF}). We also note that in both cases $T(x_{1},x_{2}) \in B((2\sigma_{1},0),\rho)$ where $\rho := C\sqrt{\delta}$ if $r \leq \delta$, and $\rho :=  \sqrt{r} + C\delta/\sqrt{r}$ if $r > \delta$. We will next infer from all of the above that $r \gtrapprox 1$. 

First, we use \eqref{form64} to deduce that $\pi(\sigma) - \pi(x_{3})$ lies at distance $\lesssim \delta$ from the point
\begin{displaymath} \pi(x_{1}) + \pi(x_{2}) = \tfrac{1}{3} \cdot \pi(T(x_{1},x_{2})) \end{displaymath}
We have just seen that all the points $\pi(T(x_{1},x_{2}))$ obtained this way (for various $x_{3} \in P(y_{1},y_{2})$) lie in an interval of length $\sim \rho$ centred at $2\sigma_{1}$. But $P(y_{1},y_{2}) \subset \mathbb{P}$ is a $(\delta,s)$-set, so
\begin{displaymath} \diam(\pi(\sigma) - \pi(P(y_{1},y_{2}))) \sim \diam(P(y_{1},y_{2})) \approx 1. \end{displaymath} 
This forces $\rho \approx 1$, hence also $r \approx 1$. In particular, we are safely outside the case $r \leq \delta$, and therefore \eqref{form65} is true for all points $T(x_{1},x_{2})$. We infer that
\begin{displaymath} \pi(x_{1}) + \pi(x_{2}) = \tfrac{1}{3} \cdot \pi(T(x_{1},x_{2})) \in \tfrac{1}{3} \cdot \pi(F \cap S_{(y_{1},y_{2})}(\mathbf{C}\delta)), \end{displaymath} 
where $\mathbf{C} \approx 1$. Using once more that $\pi(\sigma) - \pi(x_{3})$ lies at distance $\lesssim \delta$ from $\pi(x_{1}) + \pi(x_{2})$ by \eqref{form64}, we may finally conclude that 
\begin{displaymath} 3(\pi(\sigma) - \pi(x_{3})) \in [\pi(F \cap S_{(y_{1},y_{2})}(\mathbf{C}\delta))](C\delta) = (\pi(F_{(y_{1},y_{2})}))(C\delta), \end{displaymath}
where $C > 0$ is absolute. This is what we claimed in \eqref{form68}.

\subsection{Mapping circles to lines and concluding the proof of Theorem \ref{thm2}}\label{circlesToLines} We start by taking stock of what we have proven so far. We have constructed the following objects:
\begin{enumerate}
\item\label{P1} $P \subset \mathbb{P}$ is a $(\delta,s)$-set of cardinality $|P| \approx \delta^{-s}$,
\item\label{P2} $\mathcal{S} \subset P \times P$ is a $\delta$-separated set of cardinality $|\mathcal{S}| \approx \delta^{-2s}$.
\item\label{P3} $F \subset \R^{2}$ is a set with $|F|_{\delta} \lessapprox \delta^{-2s}$.
\item\label{P4} For every $(y_{1},y_{2}) \in \mathcal{S}$, the intersection $F \cap S_{(y_{1},y_{2})}(\mathbf{C}\delta)$ contains a $(\delta,s)$-set $F_{(y_{1},y_{2})}$ of cardinality $\approx \delta^{-s}$, where $\mathbf{C} \approx 1$. Here $S_{(y_{1},y_{2})}$ is the circle 
\begin{displaymath} S_{(y_{1},y_{2})} = \partial B\left((2\sigma_{1},0),\sqrt{6\sigma_{2} - 2\sigma_{1}^{2}}\right), \quad (\sigma_{1},\sigma_{2}) = y_{1} + y_{2} + Y. \end{displaymath}
\item\label{P5} $6\sigma_{2} - 2\sigma_{1}^{2} \approx 1$ for all $(y_{1},y_{2}) \in \mathcal{S}$.
\end{enumerate}
In property \eqref{P5} it should be understood that $\sigma_{1},\sigma_{2}$ refer to the coordinates of $y_{1} + y_{2} + Y$. For future reference, we add one more item:
\begin{itemize}
\item The conclusion of \eqref{P4} holds if $F$ is replaced by $F \cap \He := F \cap \{(x,y) : y \geq 0\}$.
\end{itemize}
Indeed, if we are so unlucky that most of $F \cap S_{(y_{1},y_{2})}(\mathbf{C}\delta)$ is contained in $\C \, \setminus \, \He$, for most $(y_{1},y_{2}) \in \mathcal{S}$, then we simply replace $F$ by the set $F' := \{(x,-y) : (x,y) \in F\}$. We also recall here that the notation $A \lessapprox B$ means: $A \leq C\delta^{-C\epsilon}B$, where $\epsilon > 0$ was the small parameter fixed in the previous section, and $C > 0$ is absolute.

This is all the data from the previous section we need to complete the proof of Theorem \ref{thm2}. The moral is: the family of all circles centred along the $x$-axis has the same incidence geometric properties as the family $\mathcal{A}(2,1)$. Indeed, there is an explicit map $\mathcal{G} \colon \C \to \C$ which sends circles centred along the $x$-axis to \emph{chords} of $B(1)$. I warmly thank Josh Zahl for finding this map for us! In retrospect, this map is the one which transforms the \emph{Poincar\'e half-plane model} in hyperbolic geometry to the \emph{Beltrami-Klein (disc) model}.

Roughly speaking, the set of circles $S_{(y_{1},y_{2})}$, $(y_{1},y_{2}) \in \mathcal{S}$, is a "$(\delta,2s)$-set of circles", because the parameter set $\mathcal{S} \subset P \times P$ is a $(\delta,2s)$-set. To be more precise, we will show the set of chords $\mathcal{G}(S_{(y_{1},y_{2})} \cap \{y \geq 0\})$, $(y_{1},y_{2}) \in \mathcal{S}$, spans a $(\delta,2s)$-subset of $\mathcal{A}(2,1)$. 

If the reader finds plausible what we wrote above, then he may believe that (after the transformation by $\mathcal{G}$), the set $F$ appearing in properties \eqref{P3}-\eqref{P4} is essentially a $(s,2s)$-Furstenberg set with $|F|_{\delta} \lessapprox \delta^{-2s}$. Such a set should not exist by Theorem \ref{DBZ}, and this contradiction will eventually conclude the proof of Theorem \ref{main}. 

We then turn to the details. We spell out the map $\mathcal{G}$ immediately. It is a composition of the form $\mathcal{G} = \mathcal{F} \circ \mathcal{C}$, where (in complex notation)
\begin{displaymath} \mathcal{C}(z) = \frac{z - i}{z + i} \quad \text{and} \quad \mathcal{F}(z) = \frac{2z}{1 + |z|^{2}}. \end{displaymath}
See Figure \ref{fig2}. The M\"obius map $\mathcal{C}$ is the \emph{Cayley transform}. Every M\"obius map sends circles to circles or lines, and $\mathcal{C}$ sends the $x$-axis to the unit circle $S^{1}$. Since $\mathcal{C}(i) = 0$, one sees that $\mathcal{C}$ maps the upper half-plane $\mathbb{H} := \{(x,y) : y \geq 0\}$ to the closed unit disc $B(1)$. Circles along the $x$-axis are mapped to circles intersecting $S^{1}$ twice in straight angles. A slightly special case occurs when a circle $S = S(x,r)$, $x \in \R$, contains the singularity $z = -i$ of $\mathcal{C}$: then $S$ also contains the point $z = i$, and $\mathcal{C}(S)$ is a line passing through $C(i) = 0$.
\begin{figure}[h!]
\begin{center}
\begin{overpic}[scale = 0.8]{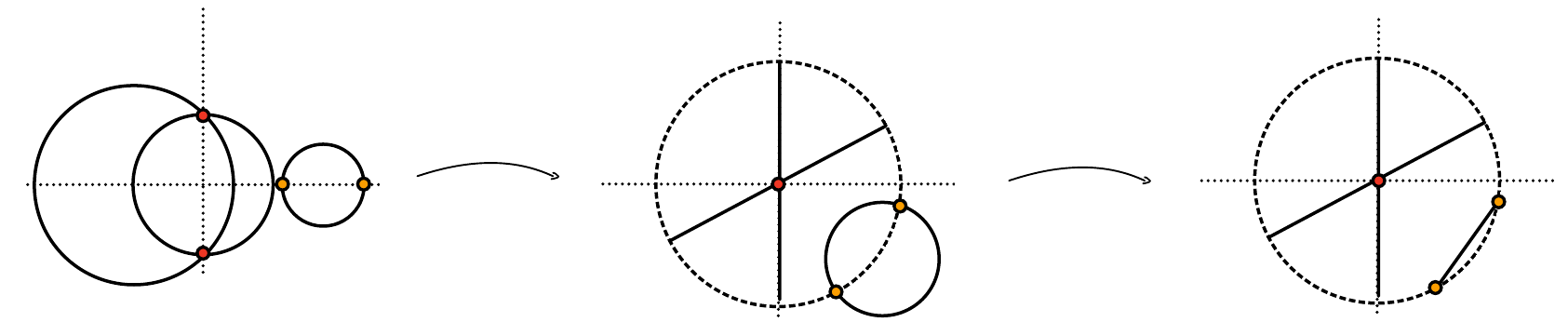}
\put(14,14){$i$}
\put(14,2){$-i$}
\put(31,11){$\mathcal{C}$}
\put(68.5,11){$\mathcal{F}$}
\put(48,9.5){$0$}
\put(86,9.5){$0$}

\end{overpic}
\caption{The maps $\mathcal{C}$ and $\mathcal{F}$.}\label{fig2}
\end{center}
\end{figure}
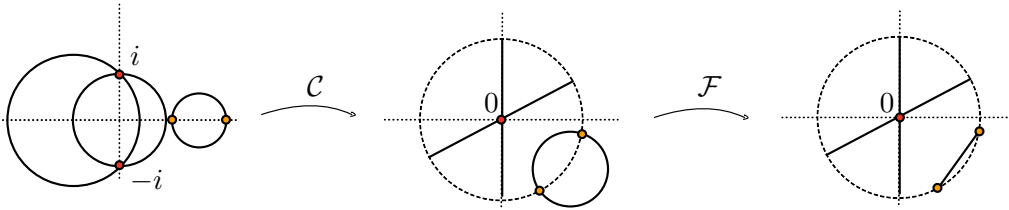

In the language of hyperbolic geometry, $\mathcal{C}$ maps the \emph{Poincar\'e half-plane model} to the \emph{Poincar\'e disc model}, where the geodesics are precisely the circles intersecting $S^{1}$ in straight angles. It is more surprising that the Poincar\'e disc model can be further mapped (by $\mathcal{F}$) to the \emph{Beltrami-Klein model}, where the geodesics are chords of $S^{1}$. This is accomplished by the map $\mathcal{F}$. It is clear from the formula that
\begin{displaymath} \mathcal{F}(B(1)) = B(1) \quad \text{and} \quad \mathcal{F}|_{S^{1}} = \mathrm{id}. \end{displaymath}
It is a bit less easy to see that the (unique) circular arc intersecting $\{a,b\} \subset S^{1}$ in straight angles gets mapped to the chord $[a,b] \subset B(1)$ under $\mathcal{F}$. This is a standard fact of hyperbolic geometry, but it was not easy to find a simple (fully) geometric argument, so we provide one in Appendix \ref{appA}.

It is clear that $\mathcal{C}$ is bilipschitz in any bounded subset of $\He$, and $\mathcal{F}$ is certainly bilipschitz on the image $C(\He) = B(1)$. So, the composition $\mathcal{G} = \mathcal{F} \circ \mathcal{C}$ is also bilipschitz on any bounded subset of $\He$. This implies that the neighbourhoods $S_{(y_{1},y_{2})}(\mathbf{C}\delta) \cap \He$ (see \eqref{P4}) are mapped to $\mathbf{C}'\delta$-neighbourhoods of chords inside $B(1)$, for some $\mathbf{C}' \approx 1$. In particular,
\begin{equation}\label{form87} \mathcal{G}(S_{(y_{1},y_{2})}(\mathbf{C}\delta) \cap \He) \subset \ell_{(y_{1},y_{2})}(\mathbf{C}'\delta), \end{equation}
where $\ell_{(y_{1},y_{2})} \in \mathcal{A}(2,1)$ is the line spanned by the chord $\mathcal{G}(S_{(y_{1},y_{2})} \cap \He)$. For similar reasons, it is clear that for every $(y_{1},y_{2}) \in \mathcal{S}$,
\begin{itemize}
\item the $(\delta,s)$-subset of $F \cap S_{(y_{1},y_{2})}(\mathbf{C}\delta) \cap \He$ is mapped to a $(\delta,s)$-subset of $\ell_{(y_{1},y_{2})}(\mathbf{C}'\delta)$,
\item $|\mathcal{G}(F \cap \He)|_{\delta} \lessapprox \delta^{-2s}$. 
\end{itemize}
Have we already established that $F' := \mathcal{G}(F \cap \He)$ is a $\delta$-discretised $(s,2s)$-Furstenberg set with $|F'|_{\delta} \lessapprox \delta^{-2s}$? This would violate Theorem \ref{DBZ} and conclude the proof of Theorem \ref{thm2}. Unfortunately, the most technical piece is still missing: we need to verify that the family of lines $\mathcal{L} := \{\ell_{(y_{1},y_{2})} : (y_{1},y_{2}) \in \mathcal{S}\}$, is a $(\delta,2s)$-set of cardinality $|\mathcal{L}| \gtrapprox \delta^{-2s}$. More precisely, we will show that $\mathcal{L}$ contains such a subset of lines.

We start with a few auxiliary results:

\begin{lemma}\label{lemma5} The set $\{y_{1} + y_{2} + Y : (y_{1},y_{2}) \in \mathcal{S}\}$ contains a $(\delta,2s)$-set $\Sigma$ of cardinality $\approx \delta^{-2s}$.
\end{lemma}

\begin{proof} Recall that $P \subset \mathbb{P}$ is a $(\delta,s)$-set, and $\mathcal{S} \subset P \times P$ has $|\mathcal{S}| \approx \delta^{-2s}$. Because $P$ is a $(\delta,s)$-set, every arc $J \subset \mathbb{P}$ of length $\mathcal{H}^{1}(J) \leq \delta^{C\epsilon}$ satisfies $|P \cap J| \lessapprox \delta^{C\epsilon s} \cdot \delta^{-s}$. In particular, if "$C$" here is chosen appropriately, at most $\tfrac{1}{2}|\mathcal{S}|$ pairs in $\mathcal{S}$ are contained in $(P \cap J) \times (P \cap J)$ for some fixed arc $J \subset \mathbb{P}$ of length $\leq \delta^{C\epsilon}$. This implies that we may find two arcs $J_{1},J_{2} \subset \mathbb{P}$ such that $\dist(J_{1},J_{2}) \approx 1$, $|P \cap J_{i}| \approx \delta^{-s}$, and
\begin{displaymath} |\mathcal{S} \cap (J_{1} \times J_{2})| \approx \delta^{-2s}. \end{displaymath}
Now, the map $g \colon (y_{1},y_{2}) \mapsto y_{1} + y_{2} + Y$ is $(\approx 1)$-bilipschitz on $J_{1} \times J_{2}$, and $\mathcal{S} \cap (J_{1} \times J_{2})$ is a $(\delta,2s)$-set. The image of $\mathcal{S} \cap (J_{1} \times J_{2})$ under "$g$" is a $(\delta,2s)$-set contained in $\{y_{1} + y_{2} + Y : (y_{1},y_{2}) \in \mathcal{S}\}$, which is denoted "$\Sigma$" from now on. \end{proof}

\begin{lemma}\label{lemma6} Let $\theta > 0$, and let $\Omega_{\theta} \subset B(1) \subset \R^{2}$ be the set
\begin{displaymath} \Omega_{\theta} := \{\sigma = (\sigma_{1},\sigma_{2}) \in B(10) : \sqrt{6\sigma_{2} - 2\sigma_{1}^{2}} \geq \theta\}. \end{displaymath}
The map 
\begin{displaymath} (\sigma_{1},\sigma_{2}) \mapsto \Phi(\sigma_{1},\sigma_{2}) = \left(2\sigma_{1} - \sqrt{6\sigma_{2} - 2\sigma_{1}^{2}},2\sigma_{1} + \sqrt{6\sigma_{2} - 2\sigma_{1}^{2}} \right) \end{displaymath}
is $O(\theta^{-1})$-bilipschitz on $\Omega_{\theta}$.
\end{lemma}

The point of this technical lemma is that the $(\delta,2s)$-set $\Sigma$ found in Lemma \ref{lemma5} is contained in $\Omega_{\theta}$ for some $\theta \approx 1$ by the property \eqref{P5} listed at the head of the section. The map $\Phi$ encodes the two intersection points of the circle $S_{\sigma_{1},\sigma_{2}}$ with the $x$-axis.

\begin{proof}[Proof of Lemma \ref{lemma6}] It suffices to show that the map
\begin{displaymath} (\sigma_{1},\sigma_{2}) \mapsto \Psi(\sigma_{1},\sigma_{2}) = \left(2\sigma_{1},\sqrt{6\sigma_{2} - 2\sigma_{1}^{2}} \right) \end{displaymath}
is $O(\theta^{-1})$-bilipschitz on $\Omega_{\theta}$, because $\Phi$ is obtained by composing $\Psi$ with the globally bilipschitz map $(x,y) \mapsto (x - y,x + y)$. Regarding $\Psi$, the whole argument is based on writing
\begin{equation}\label{form71a} \left| \sqrt{6\sigma_{2} - 2\sigma_{1}^{2}} - \sqrt{6\eta_{2} - 2\eta_{1}^{2}} \right| = \frac{|6(\sigma_{2} - \eta_{2}) + 2(\sigma_{1}^{2} - \eta_{1}^{2})|}{\sqrt{6\sigma_{2} - 2\sigma_{1}^{2}} + \sqrt{6\eta_{2} - 2\eta_{1}^{2}}}. \end{equation}
The $O(\theta^{-1})$-Lipschitz property on $\Omega_{\theta}$ follows immediately. For the co-Lipschitz estimate, split into cases where $|\sigma_{1} - \eta_{2}| \sim |(\sigma_{1},\sigma_{2}) - (\eta_{1},\eta_{2})|$, and the opposite case. In the first case, observe that 
\begin{displaymath} |\Psi(\sigma_{1},\sigma_{2}) - \Psi(\eta_{1},\eta_{2})| \geq |\sigma_{1} - \eta_{1}| \sim |(\sigma_{1},\sigma_{2}) - (\eta_{1},\eta_{2})|. \end{displaymath}
In the opposite case, observe that $|\sigma_{1}^{2} - \eta_{1}^{2}| \ll |\sigma_{2} - \eta_{2}|$, and use \eqref{form71a}. \end{proof}
\begin{cor}\label{cor1} The set 
\begin{displaymath} \Phi(\Sigma) = \left\{\left(2\sigma_{1} - \sqrt{6\sigma_{2} - 2\sigma_{1}^{2}},2\sigma_{1} + \sqrt{6\sigma_{2} - 2\sigma_{1}^{2}} \right) : (\sigma_{1},\sigma_{2}) \in \Sigma\right\} \end{displaymath}
is a $(\delta,2s)$-set of cardinality $\approx \delta^{-2s}$. \end{cor} 

\begin{proof} As discussed just before the proof of Lemma \ref{lemma6}, the set $\Sigma$ is contained in $\Omega_{\theta}$ for some $\theta \approx 1$. Thus $\Phi$ is $(\approx 1)$-bilipschitz on $\Sigma$, and such maps preserve $(\delta,2s)$-sets. \end{proof}

We record here that 
\begin{equation}\label{form71} \Phi(\Sigma) \subset \{(\xi_{1},\xi_{2}) \in [-10,10]^{2} : \xi_{2} - \xi_{1} \geq c\delta^{C\epsilon}\} =: [-10,10]^{2} \, \setminus \, \bigtriangleup, \end{equation}
where $c,C > 0$ are absolute constants. Indeed, recall that $6\sigma_{2} - 2\sigma_{1}^{2} \gtrapprox 1$ for $(y_{1},y_{2}) \in \mathcal{S}$ by \eqref{P5}. On the other hand, since $\Sigma \subset P + P + P \subset B(3)$, we have $\Phi(\Sigma) \subset [-10,10]^{2}$.

We recap what $\mathcal{G}$ does to circles centred on the $x$-axis. Every such circle is uniquely determined by its two intersection points with the $x$-axis (denoted $\R$). For $\xi_{1}, \xi_{2} \in \R$, let $S(\xi_{1},\xi_{2}) \subset \R^{2}$ be the circle centred at the $x$-axis with intersection points $\xi_{1},\xi_{2}$. Then, $\mathcal{C}$ first maps $S(\xi_{1},\xi_{2})$ to the circle $S'(\xi_{1},\xi_{2})$ which intersects $S^{1}$ in straight angles at the two points
\begin{displaymath} \mathcal{C}(\xi_{j}) = \frac{\xi_{j} - i}{\xi_{j} + i}, \qquad j \in \{1,2\}. \end{displaymath}
Next, $\mathcal{F}$ sends $S'(\xi_{1},\xi_{2}) \cap B(1)$ to the chord between $\mathcal{C}(\xi_{1})$ and $\mathcal{C}(\xi_{2})$. Consequently,
\begin{equation}\label{form88} \mathcal{G}(S(\xi_{1},\xi_{2}) \cap \He) = [\mathcal{C}(\xi_{1}),\mathcal{C}(\xi_{2})]. \end{equation}
In the proof below, it will be useful to keep in mind that $\mathcal{C}$ is bilipschitz $[-10,10] \to \mathcal{C}([-10,10]) \subset S^{1}$. In particular, if $(\xi_{1},\xi_{2}) \in [-10,10]^{2} \, \setminus \, \bigtriangleup$, recall \eqref{form71}, then $[\mathcal{C}(\xi_{1}),\mathcal{C}(\xi_{2})]$ is a chord of length $\approx 1$.
\begin{lemma}\label{lemma7} For $(\xi_{1},\xi_{2}) \in \R^{2}$ with $\xi_{1} \neq \xi_{2}$, let 
\begin{displaymath} \ell(\xi_{1},\xi_{2}) := \spa([\mathcal{C}(\xi_{1}),\mathcal{C}(\xi_{2})]) \in \mathcal{A}(2,1) \end{displaymath}
be the unique line containing the chord $[\mathcal{C}(\xi_{1}),\mathcal{C}(\xi_{2})]$. The map $(\xi_{1},\xi_{2}) \mapsto \ell(\xi_{1},\xi_{2}) \in \mathcal{A}(2,1)$ is $(\approx 1)$-bilipschitz on the set $[-10,10]^{2} \, \setminus \, \bigtriangleup$ introduced in \eqref{form71}. \end{lemma}

\begin{proof} The inequality 
\begin{displaymath} d_{\mathcal{A}(2,1)}(\ell(\xi_{1},\xi_{2}),\ell(\bar{\xi}_{1},\bar{\xi}_{2})) \lesssim |(\xi_{1},\xi_{2}) - (\bar{\xi}_{1},\bar{\xi}_{2})| \end{displaymath} 
is straightforward, and in fact holds for all $(\xi_{1},\xi_{2}),(\bar{\xi}_{1},\bar{\xi}_{2}) \in [-10,10]^{2}$. This only uses the fact that $\mathcal{C}$ is a Lipschitz map, and we leave the details to the reader. The trickier task is to prove that 
\begin{equation}\label{form73} d_{\mathcal{A}(2,1)}(\ell(\xi_{1},\xi_{2}),\ell(\bar{\xi}_{1},\bar{\xi}_{2})) \gtrapprox |(\xi_{1},\xi_{2}) - (\bar{\xi}_{1},\bar{\xi}_{2})| \end{equation}
for all $(\xi_{1},\xi_{2}),(\bar{\xi}_{1},\bar{\xi}_{2}) \in [-10,10]^{2} \, \setminus \, \bigtriangleup$. This is clear if the left hand side is $\gtrapprox 1$, so we may assume that
\begin{equation}\label{form76} r := d_{\mathcal{A}(2,1)}(\ell(\xi_{1},\xi_{2}),\ell(\bar{\xi}_{1},\bar{\xi}_{2})) \leq c_{1}\delta^{C_{1}\epsilon} \end{equation}
for suitable absolute constants $c_{1},C_{1} > 0$, to be determined in the course of the proof.

The key geometric observation is this: if $(\xi_{1},\xi_{2}) \in [-10,10]^{2} \, \setminus \, \bigtriangleup$, and $r \in (0,1]$, then
\begin{equation}\label{form72} [\ell(\xi_{1},\xi_{2})](r) \cap S^{1} \subset B(\mathcal{C}(\xi_{1}),\mathbf{C}r) \cup B(\mathcal{C}(\xi_{2}),\mathbf{C}r), \end{equation}
where $\mathbf{C} \approx 1$. This is because $[\mathcal{C}(\xi_{1}),\mathcal{C}(\xi_{2})] \subset B(1)$ is a chord of length $\approx 1$ for $(\xi_{1},\xi_{2}) \in [-10,10]^{2} \, \setminus \, \bigtriangleup$, and such chords intersect $S^{1}$ at angle $\approx 1$.

Now, let $(\xi_{1},\xi_{2}),(\bar{\xi}_{1},\bar{\xi}_{2}) \in [-10,10]^{2} \, \setminus \, \bigtriangleup$, and write $\ell := \ell(\xi_{1},\xi_{2})$ and $\bar{\ell} := \ell(\bar{\xi}_{1},\bar{\xi}_{2})$. Thus $r = d_{\mathcal{A}(2,1)}(\ell,\bar{\ell})$. This implies that 
\begin{displaymath} [\mathcal{C}(\bar{\xi}_{1}),\mathcal{C}(\bar{\xi}_{2})] \subset \bar{\ell} \cap B(1) \subset \ell(Cr) \end{displaymath}
for some absolute constant $C > 0$. As we mentioned all the way back in \eqref{form78}, the inclusion $\bar{\ell} \cap B(1) \subset \ell(Cr)$ is the only property of the metric $d_{\mathcal{A}(1,2)}$ we explicitly need in the paper. In particular,
\begin{equation}\label{form74} \{\mathcal{C}(\bar{\xi}_{1}),\mathcal{C}(\bar{\xi}_{2})\} \subset [\ell(\xi_{1},\xi_{2})](Cr) \cap S^{1} \stackrel{\eqref{form72}}{\subset} B(\mathcal{C}(\xi_{1}),\mathbf{C}Cr) \cup B(\mathcal{C}(\xi_{2}),\mathbf{C}Cr), \end{equation}
using \eqref{form72}. Formally speaking, this is possible in the following $4$ ways:
\begin{itemize}
\item[(G)] $\mathcal{C}(\bar{\xi}_{1}) \in B(\mathcal{C}(\xi_{1}),\mathbf{C}Cr)$ and $\mathcal{C}(\bar{\xi}_{2}) \in B(\mathcal{C}(\xi_{2}),\mathbf{C}Cr)$, or
\item[(B1)] $\{\mathcal{C}(\bar{\xi}_{1}),\mathcal{C}(\bar{\xi}_{2})\} \subset B(\mathcal{C}(\xi_{1}),\mathbf{C}Cr)$, or
\item[(B2)] $\{\mathcal{C}(\bar{\xi}_{1}),\mathcal{C}(\bar{\xi}_{2})\} \subset B(\mathcal{C}(\xi_{2}),\mathbf{C}Cr)$, or
\item[(B3)] $\mathcal{C}(\bar{\xi}_{2}) \in B(\mathcal{C}(\xi_{1}),\mathbf{C}Cr)$ and $\mathcal{C}(\bar{\xi}_{1}) \in B(\mathcal{C}(\xi_{2}),\mathbf{C}Cr)$.
\end{itemize}
The case (G) is good: it implies that
\begin{equation}\label{form75} |\mathcal{C}(\bar{\xi}_{j}) - \mathcal{C}(\xi_{j})| \lessapprox r, \qquad j \in \{1,2\}. \end{equation}
Since $\mathcal{C}$ is bilipschitz on $[-10,10]$, this gives $|\bar{\xi}_{j} - \xi_{j}| \lessapprox r$ for $j \in \{1,2\}$, and therefore the proof of \eqref{form73} is complete. So, it remains to show that the bad scenarios (B1)-(B3) cannot occur. In cases (B1)-(B2), we have $|\bar{\xi}_{1} - \bar{\xi}_{2}| \lessapprox \mathbf{C}Cr$, which is impossible by $(\bar{\xi}_{1},\bar{\xi}_{2}) \in [-10,10]^{2} \, \setminus \, \bigtriangleup$, assuming that constants $c_{1},C_{2} > 0$ in the the upper bound for "$r$" were chosen correctly in \eqref{form76} (relative to the constants in the definition of $\bigtriangleup$). To see that the scenario (B3) is also impossible, write
\begin{equation}\label{form77} \bar{\xi}_{2} - \bar{\xi}_{1} = (\bar{\xi}_{2} - \xi_{1}) + (\xi_{1} - \xi_{2}) + (\xi_{2} - \bar{\xi}_{1}). \end{equation}
The middle term is  negative with absolute value $\approx 1$ (since $(\xi_{1},\xi_{2}) \in [-10,10]^{2} \, \setminus \, \bigtriangleup$), and the two other terms have absolute value $\lesssim \mathbf{C}r$ in scenario (B3). Therefore, again, if the upper bound for "$r$" was chosen small enough at \eqref{form76}, we see that the right hand side of \eqref{form77} is negative in case (B3). In particular $\bar{\xi}_{2} < \bar{\xi}_{1}$, violating the assumption $(\bar{\xi}_{1},\bar{\xi}_{2}) \in [-10,10]^{2} \, \setminus \, \bigtriangleup$. This proves that only scenario (G) is possible, and completes the proof of the lemma. \end{proof} 

We can finally conclude that the set of lines $\{\spa(\mathcal{G}(S_{(y_{1},y_{2})} \cap \He)) : (y_{1},y_{2}) \in \mathcal{S}\}$ contains a $(\delta,2s)$-set of cardinality $\approx \delta^{-2s}$, namely the set $\{\spa(\mathcal{G}(S_{\sigma} \cap \He)) : \sigma \in \Sigma\}$. For this final stretch, recall the set $\Sigma \subset \R^{2}$, which was a $(\delta,2s)$-subset of $\{y_{1} + y_{2} + Y : (y_{1},y_{2}) \in \mathcal{S}\}$ of cardinality $|\Sigma| \approx \delta^{-2s}$. Recall that every $(\sigma_{1},\sigma_{2}) \in \Sigma$ is associated to the circle $S_{\sigma_{1},\sigma_{2}}$ centred along the $x$-axis. Recall that $\mathcal{G} = \mathcal{F} \circ \mathcal{C}$ sends the intersection of each such circle with $\He$ to a chord of $S^{1}$.

\begin{cor}\label{cor2} The set
\begin{displaymath} \mathcal{G}(\Sigma) := \{\spa(\mathcal{G}(S_{\sigma} \cap \He)) : \sigma \in \Sigma\} \subset \mathcal{A}(2,1) \end{displaymath} 
is a $(\delta,2s)$-set of lines with $|\mathcal{G}(\Sigma)| \approx \delta^{-2s}$. \end{cor} 

\begin{proof} In Corollary \ref{cor1}, we already observed that $\Phi(\Sigma)$ is a $(\delta,2s)$-set with $|\Phi(\Sigma)| \approx \delta^{-2s}$, and in \eqref{form71} we recorded that $\Phi(\Sigma) \subset [-10,10]^{2} \, \setminus \, \bigtriangleup$. Now, we claim that
\begin{equation}\label{form79} \spa(\mathcal{G}(S_{\sigma} \cap \He)) = \ell(\Phi(\sigma)), \qquad \sigma \in \Sigma, \end{equation}
where "$\ell$" is the map from Lemma \ref{lemma7}. This will complete the proof of the corollary, since the map "$\ell$" was shown in Lemma \ref{lemma7} to be $\approx 1$-bilipschitz on the set $[-10,10]^{2} \, \setminus \, \bigtriangleup$, and in particular on $\Phi(\Sigma)$.

The proof of \eqref{form79} is a matter of unwrapping the definitions. The circle $S_{\sigma}$ intersects the $x$-axis in precisely the two points
\begin{displaymath} \xi_{1} := 2\sigma_{1} - \sqrt{6\sigma_{2} - 2\sigma_{1}^{2}} \quad \text{and} \quad \xi_{2} := 2\sigma_{1} + \sqrt{6\sigma_{2} - 2\sigma_{1}^{2}}, \end{displaymath}
which are also the coordinates of $\Phi(\sigma)$. Thus $S_{\sigma} = S(\xi_{1},\xi_{2})$ in the notation of \eqref{form88}. Therefore, recalling the definition of $\ell(\xi_{1},\xi_{2})$ from Lemma \ref{lemma7}, we have
\begin{displaymath} \ell(\Phi(\sigma)) = \ell(\xi_{1},\xi_{2}) = \spa([\mathcal{C}(\xi_{1}),\mathcal{C}(\xi_{2})]) \stackrel{\eqref{form88}}{=} \spa(\mathcal{G}(S(\xi_{1},\xi_{2}) \cap \He)) = \spa(\mathcal{G}(S_{\sigma} \cap \He)). \end{displaymath}
This completes the proof of \eqref{form79}. \end{proof}

We finally summarise the proof of Theorem \ref{thm2}:

\begin{proof}[Proof of Theorem \ref{thm2}] The map $\mathcal{G}$ sends $S_{(y_{1},y_{2})} \cap \He$, $(y_{1},y_{2}) \in \mathcal{S}$, to a certain chord of $B(1)$, which then spans a line $\ell_{(y_{1},y_{2})} \in \mathcal{A}(2,1)$. The set of lines so obtained contains a $(\delta,2s)$-set $\mathcal{L}$ of cardinality $|\mathcal{L}| = |\Sigma| \approx \delta^{-2s}$. This is the content of Corollary \ref{cor2}. 

On the other hand, $\mathcal{G}$ is bilipschitz on bounded subsets of $\He$, so 
\begin{displaymath} \mathcal{G}(S_{(y_{1},y_{2})}(\mathbf{C}\delta) \cap \He) \subset \ell_{(y_{1},y_{2})}(\mathbf{C}'\delta), \qquad (y_{1},y_{2}) \in \mathcal{S}, \end{displaymath}
for some $\mathbf{C}' \sim \mathbf{C}$. In particular, $\ell_{(y_{1},y_{2})}(\mathbf{C}'\delta)$, $(y_{1},y_{2}) \in \mathcal{S}$, contains the $(\delta,s)$-set $\mathcal{G}(F_{(y_{1},y_{2})})$. Recall that $F_{(y_{1},y_{2})} \subset S_{(y_{1},y_{2})}(\mathbf{C}\delta) \cap F \cap \He$ was a $(\delta,s)$-set of cardinality $|F_{(y_{1},y_{2})}| \approx \delta^{-s}$, see \eqref{form103}, and the remark below \eqref{P5} (about why we may add the intersection with "$\He$").

Therefore, $F' = \mathcal{G}(F \cap \He) \subset \R^{2}$ is a $\delta$-discretised $(s,2s)$-Furstenberg set. The fact that $\mathcal{G}(F_{(y_{1},y_{2})})$ is contained in $\ell_{(y_{1},y_{2})}(\mathbf{C}'\delta)$, rather than $\ell_{(y_{1},y_{2})}(\delta)$, makes no difference: each of the thicker neighbourhoods can be covered by $\approx 1$ thinner neighbourhoods, and the ensuing slightly larger family of lines is still a $(\delta,2s)$-set. Since $|F'|_{\delta} \lesssim |F|_{\delta} \lessapprox \delta^{-2s}$, existence of $F'$ violates Theorem \ref{DBZ}, assuming that $\epsilon > 0$ was small enough, depending only on $s \in (0,1)$. This contradiction completes the proof of Theorem \ref{thm2}. \end{proof}

\appendix

\section{Mapping circular arcs to chords}\label{appA}

We give a short geometric argument for the fact that $z \mapsto 2z/(1 + |z|^{2})$ maps the Poincar\'e disc model to the Beltrami-Klein model.

\begin{proposition}\label{prop2} The map $\mathcal{F}(z) = 2z/(1 + |z|^{2})$ has the following property. Let $S \subset \R^{2}$ be a circle which intersects the unit circle $S^{1}$ in straight angles. Let $J := S \cap B(1)$ be the part of $S$ inside the closed unit disc, and let $\{a,b\} := S \cap S^{1}$. Then $\mathcal{F}(J) = [a,b]$.
\end{proposition}
\begin{figure}[h!]
\begin{center}
\begin{overpic}[scale = 1.3]{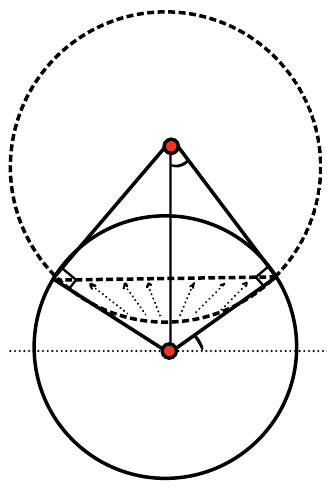}
\put(5,73){$S$}
\put(43,29.5){$\theta$}
\put(36.5,61){$\theta$}
\put(48,32.5){$1$}
\put(35,48){$\tfrac{1}{\sin \theta}$}
\put(47,58){$\tfrac{\cos \theta}{\sin \theta}$}
\put(32.5,22){$0$}
\put(26,75.5){$(0,\tfrac{1}{\sin \theta})$}
\put(5,42){$a$}
\put(60,42){$b$}
\end{overpic}
\caption{Objects in Proposition \ref{prop2}.}\label{fig1}
\end{center}
\end{figure}
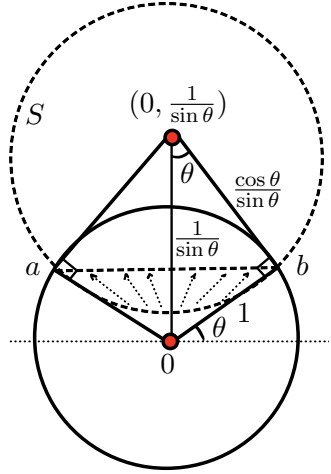

\begin{proof} It evidently suffices to consider the case where the centre of the circle $S$ lies on the $y$-axis, as in Figure \ref{fig1}. Instead of checking that the map $\mathcal{F}$ does the right thing, we "find" it as follows. We seek a map of the form $\mathcal{F}(z) = r(z)z$, where $r(z) \in [1,\infty)$, and which maps the arc $J$ to the chord $[a,b]$. 

Let $\theta \in (0,\pi)$ be the angle depicted in Figure \ref{fig1}. Using the hypothesis that $S$ meets $S^{1}$ in straight angles, one calculates that the centre of $S$ is the point $x := (0,\tfrac{1}{\sin \theta})$, and the radius of $S$ is $r := \tfrac{\cos \theta}{\sin \theta}$. Moreover, the chord $[a,b]$ is contained in the set $\{y = \sin \theta\}$.

Every point on $S$, and in particular $J$, has the form 
\begin{equation}\label{form80} z = x + re = \left(\tfrac{\cos \theta}{\sin \theta} e_{1},\tfrac{1}{\sin \theta} + \tfrac{\cos \theta}{\sin \theta} e_{2} \right), \qquad e = (e_{1},e_{2}) \in S^{1}. \end{equation}
Our desired map $\mathcal{F}(z) = r(z)z$ has the property of sending each $z \in J$ inside the set $\{y = \sin \theta\}$. Looking at the $2^{nd}$ coordinate of $z$ in \eqref{form80}, this gives
\begin{displaymath} r(z) \left( \tfrac{1}{\sin \theta} + \tfrac{\cos \theta}{\sin \theta} e_{2} \right) = \sin \theta \quad \Longleftrightarrow \quad r(z) = \tfrac{\sin^{2}\theta}{1 + e_{2} \cos \theta}.  \end{displaymath}
On the other hand, a straightforward computation, using \eqref{form80} and the identities $e_{1}^{2} + e_{2}^{2} = 1 = \cos^{2}\theta + \sin^{2}\theta$, shows that
\begin{displaymath} \frac{2}{1 + |z|^{2}} = \frac{\sin^{2}\theta}{1 + e_{2}\cos \theta} = r(z). \end{displaymath}
Thus, $z \mapsto 2z/(1 + |z|^{2})$ maps $J$ to the chord $[a,b]$, as claimed. \end{proof}

\bibliographystyle{plain}
\bibliography{references}

\end{document}